\documentclass{amsart}
\usepackage{latexsym, amsfonts, amsthm, amsmath, amssymb, amscd, epsfig, amsrefs}
\usepackage[all]{xy}
\usepackage{color}
\usepackage{mathrsfs}
\allowdisplaybreaks[1]
\usepackage{esint}

\newtheorem{cor}{Corollary}
\newtheorem{lemma}{Lemma}[section]
\newtheorem{Add in proof}{Add in proof}[section]
\newtheorem{prop}{Proposition}[section]

\newtheorem{remark}{Remark}
\newtheorem{theorem}{Theorem}

\newcommand{\R}{\ensuremath{\mathbb{R}}}
\def \p{\partial}
\def\intave#1{-\kern-10.7pt\int_{\,#1}}
\def\<{\langle}   \def\>{\rangle}
\def\({\left(}    \def\){\right)}

\def\limsup{\operatornamewithlimits{lim\,sup}}
\newcommand{\loc}{\ensuremath\mathrm{loc}}
\let\d\relax\DeclareMathOperator{\d}{d}
\renewcommand{\d}{\,d}
\DeclareMathOperator{\curl}{curl}
\DeclareMathOperator{\tr}{tr}

\let\div\relax \DeclareMathOperator{\div}{div}

\newcommand{\na}{\ensuremath{\nabla}}
\newcommand{\D}{\ensuremath{\nabla}}
\newcommand\alabel[1]{\addtocounter{equation}{1}\tag{\theequation}\label{#1}}
\DeclareMathOperator{\dist}{dist}

\begin{document}
\title{Existence of minimizers and convergence of critical points for a  new Landau-de Gennes energy functional in nematic liquid crystals}

\numberwithin{equation}{section}
\author{Zhewen Feng and Min-Chun Hong}

\address{Zhewen Feng, Department of Mathematics, The University of Queensland\\
	Brisbane, QLD 4072, Australia}
\email{z.feng@uq.edu.au}

\address{Min-Chun Hong, Department of Mathematics, The University of Queensland\\
Brisbane, QLD 4072, Australia}
\email{hong@maths.uq.edu.au}

\begin{abstract}   The  Landau-de Gennes energy in nematic liquid crystals depends on  four  elastic constants  $L_1$, $L_2$, $L_3$, $L_4$. In the   case of $L_4\neq 0$, Ball and Majumdar (Mol. Cryst. Liq. Cryst., 2010) found an  example that  the original  Landau-de Gennes energy functional in physics  does not  satisfy a  coercivity condition, which causes a problem in mathematics to establish  existence of energy minimizers. At first, we introduce a new Landau-de Gennes energy density  with   $L_4\neq 0$, which is equivalent to the    original Landau-de Gennes density for  uniaxial tensors and  satisfies the coercivity condition for all $Q$-tensors. Secondly, we prove that   solutions of the Landau-de Gennes system can approach a solution  of the $Q$-tensor Oseen-Frank system without using energy minimizers. Thirdly, we  develop a new approach to generalize   the Nguyen and Zarnescu (Calc. Var. PDEs, 2013) convergence result to the case of non-zero elastic constants $L_2$, $L_3$, $L_4$.
     \end{abstract}

 \maketitle

\pagestyle{myheadings} \markright {A new Landau-de Gennes energy}

\section{Introduction}

 A liquid crystal is a state of matter between isotropic liquid and crystalline solid.
 	Based on the molecular positional and orientational order of liquid crystals,  there are three main types: sematic, cholesterics and  nematic.  The nematic liquid crystal is the most common type in  which the general phases are    uniaxial and   biaxial.
 	 In 1971, de Gennes \cite{De} used  $Q$-tensor order parameters  to formulate the elastic energy of liquid crystals  with  Landau's bulk  energy.
 	The  Landau-de Gennes theory has been verified in physics as a successful theory describing both uniaxial and biaxial phases in nematic liquid crystals.
 Indeed,
   Pierre-Gilles de Gennes was awarded  a Nobel prize for physics in 1991    for his discoveries  in liquid crystals and polymers.

  In the Landau-de Gennes framework,  the space of $Q$-tensors  in   the Landau-de Gennes theory is  a space of symmetric, traceless $3\times3$ matrices defined by
			\begin{equation*}
			S_0:=\left\{Q\in \mathbb M^{3\times 3}:\quad Q^T=Q, \, \mbox{tr }Q  =0\right\},
			\end{equation*}
 where  $\mathbb M^{3\times 3}$ denotes the space of $3\times 3$ matrices.
When  $Q\in S_0$ has two equal non-zero eigenvalues, a nematic liquid crystal is said to be uniaxial. When $Q$ has three unequal non-zero eigenvalues, a nematic liquid crystal is said to be  biaxial. For  material constants $a , b, c$,
we define the constant order parameter
	\[s_+:=\frac{b+\sqrt{b^2+24ac}}{4c} \]
and   denote  the identity matrix by $I$. The subspace of  uniaxial $Q$-tensors is given by
	\[S_*:=\left\{Q  \in S_0:\quad Q =s_+ (u\otimes u-\frac 13 I),\quad  u\in S^2 \right\}.\]	
In this paper, we only consider the case of positive constants $a$, $b$, $c$, which corresponds to a lower temperature regime in liquid crystals  (the constant $a$ could also be negative; see \cite{MZ}, \cite{MN}).

 Let $\Omega$ be a domain in $\R^3$. For a tensor $Q\in W^{1,2}(\Omega ; S_0)$,
	the Landau-de Gennes energy is defined by
	\begin{equation}\label{LG energy}
	E_{LG}(Q; \Omega)=\int_{\Omega}f_{LG}(Q,\nabla Q)\,dx:=\int_{\Omega}\left (  f_E(Q,\nabla Q) +  f_B(Q)\right )\,dx,
	\end{equation}
	where $ f_E$ is the elastic energy density  with elastic constants $L_1,...,L_4$ of the form
		\begin{equation}\label{LG}
	 f_E(Q,\D Q):=\frac{L_1}{2}|\D Q|^2+\frac{L_2}{2} \frac{\p Q_{ij}}{\p x_j} \frac{\p Q_{ik}}{\p x_k} +\frac{L_3}{2}\frac{\p
Q_{ik}}{\p x_j}\frac{\p Q_{ij}}{\p x_k}+\frac{L_4}{2}Q_{lk}\frac{\p Q_{ij}}{\p x_l}\frac{\p Q_{ij}}{\p x_k}
	\end{equation}
	 and $f_B(Q)$ is  the bulk  energy density defined by
	\begin{equation}\label{BE}
	 f_B(Q):=-\frac{ a}{2}\tr( Q^2)-\frac{ b}{3}\tr( Q^3)+\frac{ c}{4}\left[\tr( Q^2)\right]^2
	\end{equation}
with  positive material constants $a$, $b$ and $c$. Here and in the sequel, we adopt  the Einstein summation convention for repeated indices.

In  \cite{De}, de Gennes discovered the first  two terms of the elastic energy density in (\ref{LG}) with $L_3=L_4=0$. Since  both  the Oseen-Frank theory and the Landau-de Gennes theory should unify  for modeling  uniaxial liquid crystals,
 Schiele and Trimper  \cite{ST} pointed out  that the
early attempt of de Gennes' work \cite{De} was incomplete since it would require the splay and bend Frank constants to be equal (i.e. $k_1=k_3$) in the Oseen-Frank density (as defined below in \eqref{OF density})  of    uniaxial tensors $Q =s_+ (u\otimes u-\frac 13 I)$. However, some experiments on liquid crystals showed that $k_3>k_1$, so they added a third order  term to original de Gennes' elastic energy density  by
\[ \frac{L_1}{2}|\nabla  Q|^2+\frac{L_2}{2} \frac{\p Q_{ij}}{\p x_j} \frac{\p Q_{ik}}{\p x_k} +\frac{L_4}{2} Q_{lk}\frac{\p Q_{ij}}{\p x_l}\frac{\p Q_{ij}}{\p x_k} \]
with $L_4=\frac{1}{2s_+^3}(k_3-k_1)>0$.  Later, Berreman and Meiboom \cite{BM84} observed that the above two groups
     discarded the surface energy density in the Oseen-Frank density,
     which correlates the blue phase theory for liquid crystals, so
     they proposed to recover a second order term  $\frac{L_3}{2}\frac{\p
Q_{ik}}{\p x_j}\frac{\p Q_{ij}}{\p x_k}$ with four third order terms, but their density is over-determined with the Oseen-Frank density.
     Later, Longa et al. \cite{LMT} gave an extension of the Landau-de Gennes density with 22 independent parameters.
   Finally, combining the work of  Schiele and Trimper  \cite {ST} with Berreman and Meiboom \cite{BM},  Dickmann  \cite  {Di}  found the  full density
    \eqref{LG}, which is consistent with
 the Oseen-Frank density in (\ref {OF density}) for uniaxial  nematic liquid crystals. Since then, the general form (\ref{LG}) of the Landau-de Gennes energy density
    has been widely used in the study of nematic liquid crystals  (e.g.  \cite {MGKB},  \cite {MN},  \cite {Ba}). Under the relation that
\begin{align*}
 &s_+^2L_1 =-\frac 16k_1+\frac 12k_2+\frac 16k_3,\,
	s_+^2L_2 =k_1-k_2-k_4,\,
	 s_+^2L_3 =k_4,\,
	s_+^{3}L_4 =\frac{k_3-k_1}{2},
\end{align*}	
  one can show (c.f. \cite  {MGKB}) that for a  uniaxial tensor $Q =s_+ (u\otimes u-\frac 13 I)$ with $u\in S^2$,
\begin{align*}\alabel{WnF}
	f_E(Q,\nabla Q)=W(u,\na u),
\end{align*}
where  the Oseen-Frank density $W(u,\na u)$  is defined by
 \begin{align}\label{OF density}
			W(u,\na u) =& \frac {k_1} 2(\div u)^2+ \frac {k_2 } 2(u\cdot \curl u)^2+\frac {k_3} 2|u\times \curl u|^2\\
&+\frac {k_2+k_4 } 2(\tr(\na u)^2-(\div u)^2) \nonumber
		\end{align}
for a unit director $u\in W^{1,2} (\Omega; S^2)$.    In  \eqref {OF density}, $k_1$, $k_2$, $k_3$ are the Frank constants for molecular distortion  of splay, twist and bend  respectively and $k_4$ is the Frank constant for the surface energy  (c.f. \cite{DP}).
In 1937, Zvetkov    established
numerical values for p-azoxyanisole (PAA) at $120^\circ C$ (with the unit $10^{-12} \,m/J$) as follows:
\begin{align*}
&k_1=5,\;k_2=3.8,\;k_3=10.1.
\end{align*}
Therefore, according to physical experiments on nematic liquid crystals, the elastic constant $L_4=\frac{1}{2s_+^3}(k_3-k_1)$ is  not equal to zero in general (c.f. \cite{DP}).

A fundamental  problem  in mathematics on the  Landau-de Gennes theory is to establish existence of a minimizer  of the energy functional $E_{LG}(Q,\Omega)$ in  $W^{1,2}_{Q_0}(\Omega ; S_0)$ with $L_4\neq 0$.
If the functional density $f_{LG}(Q,\D Q)$   satisfies the  coercivity condition, one can prove   existence of a minimizer of the functional  $E_{LG}(Q,\Omega)$ in $W^{1,2}(\Omega; S_0)$. In  2010,   Ball and Majumdar  \cite {BM}  found an example   where for  $Q\in S_0$, the  general Landau-de Gennes energy density  \eqref{LG} with $L_4\neq 0$ does not  satisfy the
coercivity condition.
Very recently,  Golovaty et al. \cite{GNS20} emphasized that ``From the
standpoint of energy minimization, unfortunately, such a version of Landau-de Gennes becomes
problematic, since the inclusion of the cubic term leads to an energy which is unbounded
from below''.    Therefore, the    Landau-de Gennes density (\ref{LG}) causes a knowledge gap between mathematical and physical theories on liquid crystals,  since the energy functional  $E_{LG}(Q,\Omega)$ in $W^{1,2}(\Omega; S_0)$  does not satisfy the coercivity condition and violates the existence  theorem of minimizers   (e.g. \cite{Gi}, \cite{Ba}).  In physics,  concerning the third order term $\frac{L_4}{2} Q_{\alpha\beta}\frac{\p Q_{ij}}{\p  x_\alpha}\frac{\p Q_{ij}}{\p x_\beta}$ with $L_4\neq 0$ in \eqref{LG},  Longa et al.  \cite{LMT} questioned     that ``In the presence of biaxial fluctuations the general third order
theory in $Q_{\alpha\beta}$ becomes unstable and thus is thermodynamically incorrect".
In order to overcome the difficulty, they  extended Landau-de Gennes densities through   22 independent second, third, fourth order terms,
  to preserve the stability of the free energy.  Although their result is very interesting,   all energy densities in \cite{LMT} are complicated and have not addressed the above coercivity problem for  all $Q$-tensors. In 2020, following similar  spirit  in \cite{LMT},    Golovaty et al. \cite{GNS20}  proposed a new physical interpretation  of the density through  fourth order terms to address the above coercivity problem for  all $Q$-tensors.  We would like to point out that  the new  density form in \cite{GNS20}  is completely different  from the original  Landau-de Gennes density \eqref{LG} although the  density  in \cite{GNS20} can recover  the Oseen-Frank density for uniaxial $Q$-tensors.

 In this paper,  we will  propose a new Landau-de Gennes energy density   to solve  the above coercivity problem  with $L_4\neq 0$. At first,  we observe  in Lemma \ref{Lemma 2.1} that for uniaxial tensors $Q\in S_*$,
the original
third order  term  on  $L_4$ in  (\ref {LG}), proposed by
  Schiele and Trimper  \cite{ST}*{p.~268} in physics, is a linear combination of a fourth order term and a second order term in the following:
 \begin{align*}\alabel{third}
	Q_{lk}\frac{\p Q_{ij}}{\p x_l}\frac{\p Q_{ij}}{\p x_k}=\frac{3}{s_+}(Q_{ln}\frac{\p Q_{ij}}{\p x_l})(Q_{kn}\frac{\p Q_{ij}}{\p x_k})-\frac {2s_+}3|\na Q|^2.
	\end{align*}
In the case of $L_4\geq 0$,    we introduce a new  elastic energy density
\begin{align*}\alabel{D}
 	f_{E,1}(Q,\D Q)=&\(\frac{L_1}{2}-\frac{s_+L_4}{3}\)|\D Q|^2
 	+\frac{L_2}{2} \frac{\p Q_{ij}}{\p x_j}\frac{\p Q_{ik}}{\p x_k}
 	\\
 	&+\frac{L_3}{2} \frac{\p Q_{ik}}{\p x_j}\frac{\p Q_{ij}}{\p x_k}+  \frac{3L_4}{2s_+} Q_{ln}Q_{kn}\frac{\p Q_{ij}}{\p x_l}\frac{\p Q_{ij}}{\p x_k}
 \end{align*}for all $Q\in S_0$.
We should point out that the  fourth order term $Q_{ln}Q_{kn}\frac{\p Q_{ij}}{\p x_l}\frac{\p Q_{ij}}{\p x_k}$ in \eqref{D} is a non-negative square term, so    our  new Landau-de Gennes density    \eqref{D} for $Q\in S_0$  satisfies the coercivity condition in mathematics under suitable conditions on $L_1,\cdots, L_4$. The first three terms in \eqref{D} keep  the original  form  \eqref{LG} for $Q\in S_0$ and the new Landau-de Gennes density \eqref{D} is equivalent to the  original density \eqref{LG} for   $Q\in S_*$. We also remark that our  fourth order term $Q_{ln}Q_{kn}\frac{\p Q_{ij}}{\p x_l}\frac{\p Q_{ij}}{\p x_k}$ is a linear combination
 of  three fourth order terms $L^{(4)}_5, L^{(4)}_6, L_7^{(4)}$ in \cite{LMT}; i.e., we verify in Lemma \ref{Lemma 2.1.b} that
 \[Q_{ln}Q_{kn}\frac{\p Q_{ij}}{\p x_l}\frac{\p Q_{ij}}{\p x_k}=\frac{8}{5} L^{(4)}_5- \frac{2}{5}L^{(4)}_6 +\frac{2}{5}L_7^{(4)}.\]
For a  $Q\in W^{1,2}(\Omega, S_0)$, we introduce a new Landau-de Gennes  energy functional
\begin{equation}\label{LG1}
 E_{L}(Q; \Omega )=   \int_{\Omega}f_{LG}(Q,\nabla Q)\,dx =\int_{\Omega}\(  f_{E,1}(Q, \nabla Q) +\frac 1 L \tilde f_B(Q)\)\,dx,
 \end{equation}
 where $f_{E,1}(Q, \nabla Q)$ has the form  \eqref{D},
$\tilde f_B(Q):= f_B(Q)-\min_{Q\in S_0} f_B(Q)\geq 0$ and $L>0$ is a parameter to drive all elastic constants to zero \cites{MN,BPP,Ga}.

 Although there are  many differences between the Oseen-Frank theory and the Landau-de Gennes theory, it is of great interest in mathematics  whether  minimizers of the Landau-de Gennes  energy functional can approach a minimizer  of the Oseen-Frank energy functional.
When $L_2=L_3=L_4=0$ in (\ref{D}),  Majumdar and Zarnescu \cite {MZ} first proved that as $L\to 0$, minimizers $Q_{L}$ of $E_{L}$ converges to  $Q_*=s_+ (u^*\otimes u^*-\frac 13 I)$,  where $Q_*$ is a minimizer of the Dirichlet energy functional in $W^{1,2}_{Q_0}(\Omega ; S_*)$.   Sine then, there exist many developments on the one-constant approximation (c.f. \cite{Ba}) and some special cases of  unequal constants $L_2$, $L_3$, $L_4$ (\cite{BPP}, \cite{IXZ}). In theory of liquid crystals, the general expectation on the elastic constants is that   $L_4$ is not always zero (c.f. \cite{ST}*{p.~268}, \cite{BM}). For the case of $L_4\neq 0$,  we first prove

	\begin{theorem}\label{Theorem 1} Let $L_1$, $L_2$, $L_3$ and $L_4$ be elastic constants satisfying
\begin{align}\label{L}
	&L_1-\frac{s_+L_4}{6}>0, \quad   - L_1-\frac{s_+L_4}{6}<L_3<2 L_1-\frac{s_+L_4}{3},\\
	&\, L_1-\frac{s_+L_4}{6}+\frac53L_2+\frac16L_3>0,\quad L_4\geq 0.\nonumber
\end{align}
 Then, for each $L>0$, $f_{LG}(Q,\nabla Q)$ in \eqref{LG1}  satisfies  the
coercivity condition so that
there exists a minimizer $Q_L$ of the functional (\ref{LG1}) in $W^{1,2}_{Q_0}(\Omega; S_0 )$ with  boundary value $Q_0\in W^{1,2}(\Omega; S_* )$. As $L\to 0$, the minimizers $Q_L$ of   $E_L(Q;\Omega)$
		converge (up to a subsequence) strongly  to  $Q_*$ in $W^{1,2}_{Q_0}(\Omega; S_0)$ and satisfies
\begin{align*}
	\lim_{L\to 0}\frac 1 L\int_{\Omega}  \tilde f_B(Q_L) \,dx=0.
\end{align*}Furthermore,
 $Q_*$ is a  minimizer of  the functional
$E(Q; \Omega):=\int_{\Omega}   f_{E,1}(Q, \nabla Q)\,dx$  for all  uniaxial $Q$-tensors in  $W^{1,2}_{Q_0}(\Omega; S_* )$.
\end{theorem}

\begin{remark}
	 In the case of $L_4<0$,  for each $Q\in W^{1,2}(\Omega, S_0)$, we  introduce an  elastic energy density  by
\begin{align}
f_{E,-}(Q,\nabla Q):=& \(\frac{L_1}{2}+\frac{s_+L_4}{3}\)|\D Q|^2+\frac{L_2}{2}
\frac{\p Q_{ij}}{\p x_j} \frac{\p Q_{ik}}{\p x_k}+ \frac{L_3}{2} \frac{\p Q_{ik}}{\p x_j}\frac{\p Q_{ij}}{\p x_k}
\\
& -\frac{3}{s_+}L_4\left (|Q|^2|\D Q|^2-Q_{ln}Q_{kn}\frac{\p Q_{ij}}{\p x_l}\frac{\p Q_{ij}}{\p x_k}\right ).\nonumber\end{align}
 It is clear that  $|Q|^2|\D Q|^2-Q_{ln}\frac{\p Q_{ij}}{\p x_l}Q_{kn}\frac{\p Q_{ij}}{\p x_k}\geq 0$ for each $Q\in W^{1,2}(\Omega, S_0)$, and that the elastic energy density  $f_E(Q,\nabla Q)$ in (\ref {LG}) is equal  to 	$f_{E,-}(Q,\nabla Q)$  for uniaxial tensors $Q\in S_*$.

\end{remark}

Next, we discuss   critical points of the  Landau-de Gennes energy functional (\ref{D}) in  $W^{1,2}_{Q_0}(\Omega; S_0)$.
One can write  $f_{E}(Q,\D Q):=\frac {\alpha}2 |\D Q|^2+V(Q, \D Q)$ for some  $\alpha>0$ so that $V(Q, \D Q)\geq 0$ for all $Q\in W^{1,2}_{Q_0}(\Omega; S_0)$. Then,  the    Euler-Lagrange equation  for  the  Landau-de Gennes energy functional \eqref{LG1}  in  $W^{1,2}_{Q_0}(\Omega; S_0)\cap L^{\infty}(\Omega; S_0)$ is

	\begin{align}\label{MDEL}
		&\alpha  \Delta Q +   \frac 12 {\na_k(V_{Q_{x_k}}+  V^T_{Q_{x_ k}})}-\frac 1 3  I \mbox{ tr} (\na_k  V_{Q_{x_k}}) -\frac 1 2 (  V_{Q}  +  V^T_{Q})+   \frac 1 3 I   \mbox{ tr} (V_{Q})
		\\
		=& \frac 1 L \(-aQ - b  \(Q Q -\frac 13I \tr(Q^2)\) +cQ \tr(Q^2)\)\nonumber
	\end{align}
in the weak sense, where  $A^T$  denotes the transpose of $A$,
\[V_{Q}:=\frac {\partial V(Q,\nabla Q)}{\partial Q} \mbox{ and }  V_{Q_{x_k}}:=\frac {\partial V(Q,\nabla Q)}{\partial Q_{x_k}}.\]

  For general elastic constants $L_1,
\cdots$, $L_4$,  we cannot find any reference having an explicit  form of  the
  Euler-Lagrange equation of $E(Q; \Omega)$  for $Q=s_+ (u\otimes u-\frac 13 I)\in S_*$ with $u\in S^2$, so we   give an explicit  form of the
  Euler-Lagrange equation   in the following:
\begin{align*}\alabel{EL}
		&\quad  \alpha \(s_+ \Delta Q  -2 \na_kQ\na_kQ +2 s_+^{-1}(Q+\frac {s_+}3I)|\na Q|^2\)\\
		&+\nabla_k \(
		V_{Q_{x_k}} (Q +\frac {s_+}3 I)
		+(Q +\frac {s_+}3 I)V^T_{Q_{x_k}}
		-2s_+^{-1} (Q +\frac {s_+}3 I)\<Q+\frac {s_+}3 I,\, V_{Q_{x_k}} \>\)
		\\
		& -  V_{Q_{x_k}}\na_kQ
		-\na_k Q V^T_{Q_k}
		+2s_+^{-1}\left[\<V_{Q_k},  \nabla_k Q\>( Q +\frac {s_+}3 I)+ \<V_{Q_{x_k}},  ( Q+\frac {s_+}3 I)\>\nabla_k Q\right]
		\\
		&-  (Q +\frac {s_+}3 I) V^T_{Q}
		- V_{Q} (Q +\frac {s_+}3 I)
		+2s_+^{-1}\<V_{Q}, Q +\frac {s_+}3 I\> (Q +\frac {s_+}3 I)=0,
		\end{align*}
  which is equivalent to  the Oseen-Frank system for $u\in S^2$, where   $\<A, B\>= \tr (B^TA)$ is the standard inner product of two matrices $A$ and $B$. In   the case of  $L_2=L_3=L_4=0$,  the
  Euler-Lagrange equation \eqref{EL} reduces to
 \[s_+\Delta Q-2 \frac {\p Q}{\p x_k}  \frac {\p Q}{\p x_k}+2s_+^{-1}(Q+\frac {s_+}3I)|\na Q|^2=0, \]
which is equivalent to the harmonic map equation of $u$ (c.f. \cite{NZ}).

Since  the Landau-de Gennes theory has been successfully used    for modeling   both uniaxial   and biaxial states of nematic liquid crystals, it is of great interest  whether the  $Q$-tensor type of the Oseen-Frank system can be approximated by the Landau-de Gennes system \eqref{MDEL} without using minimizers.	In general, the problem of the convergence of solutions of the  Landau-de Gennes equation \eqref{MDEL} without using minimizers is  open. Indeed,
Gartland   \cite {Ga} pointed out   that the convergence of solutions of the  Landau-de Gennes equation \eqref{MDEL}  is  similar to the convergence of solutions of  the Ginzburg-Landau approximate equation from superconductivity theory.  The Ginzburg-Landau functional  was introduced in  \cite{GL} to study the phase transition in superconductivity.  For a  parameter $\varepsilon>0$,  the Ginzburg-Landau functional of $u:\Omega\rightarrow \R^3$ is defined by
\begin{equation}\label{GLF}
E_{\varepsilon}(u; \Omega):=\int_{\Omega}\(\frac 12 |\nabla u|^2+\frac{1}{4\varepsilon^2}(1-|u|^2)^2\)\, dx.
\end{equation}
The Euler-Lagrange equation  for the Ginzburg-Landau functional is
\begin{equation}\label{GL}
\Delta u_{\varepsilon}+\frac 1 {\varepsilon^2} u_{\varepsilon}(1-|u_{\varepsilon}|^2)=0.
\end{equation}
 Using the cross product of vectors,
Chen \cite {Chen} first proved that as $\varepsilon\to 0$,  solutions $u_{\varepsilon}$ of the Ginzburg-Landau system (\ref{GL}) weakly converge  to a harmonic map in $W^{1,2}(\Omega;  \R^3)$.  Chen and Struwe \cite{CS} proved global existence of partial regular  solutions to the heat flow of harmonic maps using the Ginzburg-Landau approximation. In  \cite{BBH}-\cite{BBH1},  Bethuel, Brezis and H\'elein obtained many results on  asymptotic behavior for minimizers of $E_{\varepsilon}$ in two dimensions  as $\varepsilon \to 0$ (see also \cite{Struwe}).  Recently, many works (\cite{Ho2}, \cite{HX}, \cite{HLX}, \cite{FHM}) have examined the convergence of  the Ginzburg-Landau approximation for the Ericksen-Leslie system with unequal Frank's constants $k_1, k_2, k_3$. Motivated by the above results on the Ginzburg-Landau approximation, it is natural  to investigate the converging problem on solutions of the Landau-de Gennes system \eqref{MDEL} as $L\to 0$. By comparing with the result of Chen \cite {Chen} (see also \cite{CHH}) on the weak convergence of solutions of the  Ginzburg-Landau equations,   it is  interesting  to study whether  solutions $Q_L$ of the Landau-de Gennes equations (\ref{MDEL}) with uniform bound of the energy  converge   weakly   to a  solution  $Q_*$  of $\eqref{EL}$ in $W^{1,2}_{Q_0}(\Omega; S_0 )$.
 However, it seems  that  the problem is not clear when  $L_2$,  $L_3$, $L_4$ are not zero.
  Under a condition, we  solve this problem  and prove:
	\begin{theorem} \label{Theorem 2}
Let $Q_L$ be a weak solution to the equation \eqref{MDEL} with a uniform bound in $L$.
Assume that the solution $Q_L$
		 converges  strongly to     $Q_*$ in $W^{1,2}_{Q_0}(\Omega; S_0 )$ as $L\to 0$ and satisfies
\begin{align} \label{eq Theorem 3}
	\lim_{L\to 0}\frac 1 L\int_{\Omega}  \tilde f_B(Q_L) \,dx=0.
\end{align}
Then, $Q_*$  is a weak solution to \eqref{EL}.

\end{theorem}

For the proof of Theorem \ref{Theorem 2}, we use a concept of a projection $\pi$ in a neighborhood   $S_{\delta}$ of the space $S_*$,
 where
\begin{align*}
	S_\delta:=\left \{Q\in S_0:\quad \dist(Q;S_*)\leq \delta\right\}.
\end{align*}
  For a sufficiently small $\delta >0$, there exists a smooth projection $\pi: S_{\delta}\to S_*$ so that for any $Q\in S_{\delta}$, $\pi(Q)\in S_*$ (c.f.  \cite{CS}). By the projection $\pi$, we  consider the modified bulk energy density \[F(Q):=\tilde f_B(Q)+|Q-\pi(Q)|^2 \]
 so that the Hessian of $F(Q)$ is positive definite for each any $Q\in S_{\delta}$ with a sufficiently small $\delta >0$.  Then, choosing a suitable test function and using \eqref{eq Theorem 3}, we employ Taylor's expansion of $F(Q)$ to cancel the limit term  involving  $\frac1L\na_{Q_{ij}} f_B(Q_L)$. With these results, we divide the domain into three parts and then employ Egoroff's theorem to prove  Theorem \ref {Theorem 2}.

When $L_2=L_3=L_4=0$, Majumdar and Zarnescu  \cite {MZ} first proved  that  minimizers $Q_L$  of $E_{LG}(Q; \Omega)$ uniformly converge  to $Q_*$    away from the singular set of $Q_*$ since there exists a monotonicity
formula for minimizers $Q_L$  of $E_{LG}(Q; \Omega)$ in $W^{1,2}(\Omega, S_0)$. Later, Nguyen and Zarnescu \cite{NZ} improved  the result by proving  local smooth convergence of   minimizers  $Q_L$ away from the singular set of $Q_*$. However, in the general   case of  non-zero elastic constants $L_2$, $L_3$ and $L_4$, there exists no such monotonicity
formula for minimizers $Q_L$ of $E_{LG}(Q; \Omega)$ in $W^{1,2}(\Omega, S_0)$, so it is a very interesting question whether  one can improve the convergence of $Q_L$ for  general cases. For this question,  we    generalize    Nguyen and Zarnescu's result in \cite{NZ} to the case of non-zero elastic constants $L_2, L_3, L_4$  as follows:
\begin{theorem} \label{Theorem 3}
For each $L>0$, let   $Q_L$ be a weak solution to the equation \eqref{MDEL}    and let $Q_*$ be  the limiting map of $Q_L$ in Theorem  \ref{Theorem 2}.    Assume that $Q_L$ is smooth and converges to $Q_*$ uniformly inside $\Omega\backslash \Sigma$, where $\Sigma$ is the singular set of $Q_*$.
 Then,   as $L\to 0$ (up to a subsequence),  we have
	\begin{equation}\label{con1}
	Q_L\to Q_*\text{ in }C^k_\loc(\Omega\backslash\Sigma)
	\end{equation} for any positive integer $k\geq 0$.
\end{theorem}

We would like to point out that our proof of Theorem \ref{Theorem 3} is new and different from one in \cite{NZ}.
We   outline main steps  as follows:

Step I.  For each $Q \in S_\delta$, there exists a rotation $R(Q)\in SO(3)$ such that   $\tilde Q=R^T(Q)QR(Q)$ is   diagonal. For any   $\xi \in S_0$, we prove
\begin{align*}\alabel{Hessian}
	\sum_{i,j=1}^3\p_{\tilde Q_{ii}}\p_{\tilde Q_{jj}} f_B(\tilde Q)\xi_{ii}\xi_{jj}\geq \frac{\lambda}{2} ( \xi_{11}^2 +  \xi_{22}^2+ \xi_{33}^2),
\end{align*}
where $\lambda =\min\{3a,s_+ b \}>0$.
   For each smooth $Q(x)\in S_0$, $R^T(Q(x))Q(x)R(Q(x))$ is   diagonal. Then there exists a measure zero set $\Omega_0$ such that $Q(x)$ has a constant multiplicity of eigenvalues inside subdomains of $\Omega\backslash \Omega_0$ and $R(Q(x))$ is almost differentiable in $\Omega$. Using the   geometric identity  \[ \nabla  \( R^T(Q)  Q R (Q)\)_{ii} =\(R^T(Q)\nabla     Q R (Q)\)_{ii}\] with $i=1, 2, 3$,
   we  apply
  \eqref{Hessian} to obtain
\begin{align*}
		&\int_{B_{r_0}(x_0)}  |\D^2 Q_L|^2\phi^2\, dx  \leq  C\int_{B_{r_0}(x_0)}|\D Q_L|^2  |\D\phi|^2\,dx,
	\end{align*}
for  a uniform constant $C$ in $L$, where $\phi$ is a cutoff function in $B_{r_0}(x_0)\subset \Omega\backslash \Sigma$.

Step II.  Using Step I
 with the  technique of `filling hole'  on elliptic systems \cite{Gi},
we  establish   a uniform Caccioppoli  inequality  for solutions $Q_L$ of \eqref{MDLG} in $L$; i.e.,
there exists a uniform constant $C$ independent of $L$ such that
 \begin{align}\label{C1}
		\int_{B_{r/2}(x_0)} |\D Q_L|^2\,dx\leq\frac C{r^2}\int_{B_{r}(x_0)} | Q_L-  Q_{L,x_0, r}|^2\,dx
	\end{align}
	for any $x_0$ with $B_{r_0}(x_0)\subset \Omega\backslash \Sigma$ and any $r\leq r_0$, where  $Q_{L,x_0, r}:=\fint_{B_{r}(x_0)}Q_L\,dx$.

Step III.  Based on the   uniform Caccioppoli inequality \eqref{C1}, we   apply   the well-known Gagliardo-Nirenberg interpolation (c.f. \cite {HM} or \cite{FHM}) to  obtain a control on the local $L^3$-estimate; i.e., there exists a uniform constant $r_0>0$ such that
	\begin{align}\label{key1}
		 \int_{B_{r_0}(x_0)} |\D Q_L|^3\,dx\leq \varepsilon_0
\end{align}
for some small $\varepsilon_0>0$. Then, combining \eqref{key1} with \eqref{Hessian}, we apply an induction method to obtain uniform estimates   on higher derivatives $\nabla^kQ_L$ in $L$ and prove
Theorem \ref{Theorem 2}.

\begin{remark} When $L_4$ is sufficiently small, the $L^p$-theory on the constant elliptic system (c.f. \cite{Gi}) can assure that the  weak solution  $Q_L$ to the equation \eqref{MDEL}    is smooth in $\Omega$.   In the case of $L_4=0$, Contreras and Lamy \cite{CL18} proved that the minimizers $Q_L$  uniformly converge to $Q_*$  in $\Omega\backslash \Sigma$    by assuming that $Q_L$ is uniformly bounded.  \end{remark}

\begin{remark}  In a recent work  \cite{FHM2} with Yu Mei,   we can expand ideas on proofs of Theorems   \ref {Theorem 2}-\ref {Theorem 3} to show that   solutions for the Beris-Edwards system for  biaxial $Q$-tensors  converge smoothly to  the solution of the Beris-Edwards system  for uniaxial $Q$-tensors up to its maximal existence time. \end{remark}

Finally, we make some  remarks about new forms of the Landau-de Gennes  energy density through a strong Ericksen's condition on the Oseen-Frank  density.
Recently,  Golovaty et al.  \cite{GNS20}  proposed a novel form of  the Landau-de Gennes  energy density through the Oseen-Frank  density. In addition to Ericksen's inequalities that \\$k_2\geq |k_4|, k_3\geq 0,  2k_1\geq k_2+k_4$,
 they also assumed that
	\begin{align*}\alabel{GNS}
		k_1 \geq k_2 + k_4,\, k_3 \geq k_2 + k_4,\,k_4 \leq 0.
	\end{align*}
	However, their result is not optimal.  In fact, Golovaty et al.  \cite{GNS20} made a further remark that their assumption \eqref{GNS} could be relaxed if one   includes  more cubic terms in \cite{LMT}, but they did not do it. In  Section 5, we  improve their result to derive an implicit form of $f_E(Q,\D Q)$. Assuming the strong Ericksen inequality with a weaker assumption that $2k_3> k_2+k_4 $  instead of the condition \eqref{GNS}, we  write the   Oseen-Frank  density into a new form, which satisfies the coercivity condition for all $u\in\R^3$.  In \cite{DP}, de Gennes and Prost remarked that the bending constant $k_3$ is much larger than others $k_1$ and  $k_2$. Therefore,    the assumption $2k_3> k_2+k_4 $  is satisfied through a strong  Ericksen's condition.
We give an explicit form of $f_E(Q,\D Q)$ in Proposition \ref{Theorem 4}, which satisfies the coercivity condition for all $Q\in S_0$.

The paper is organized as follows. In Section 2, we prove   Theorem  \ref{Theorem 1}. In Section 3, we prove   Theorem  \ref{Theorem 2}.  In Section 4, we prove Theorem \ref{Theorem 3}.
In Section 5, we obtain new forms of the Landau-de Gennes  energy density through a strong Ericksen's condition.

\section{The coercivity condition and existence of minimizers}
At first, we note
 \begin{lemma} \label{Lemma 2.1}
 For a uniaxial $Q= s_+(u\otimes u-\frac 13 I)\in S_*$ with $u\in S^2$, we have
	 	 \begin{align*}\alabel{Q third}
	 Q_{lk}\frac{\p Q_{ij}}{\p x_l}\frac{\p Q_{ij}}{\p x_k}
= \frac{3}{s_+}(Q_{ln}\frac{\p Q_{ij}}{\p x_l})(Q_{kn}\frac{\p Q_{ij}}{\p x_k})-\frac {2s_+}3|\na Q|^2.
	\end{align*}
	 \end{lemma}
\begin{proof} Using the fact that $|u|=1$, we have
	\begin{align*}\alabel{ID}
	& Q_{ln}Q_{kn}= s_+^2(u_ku_n-\frac13 \delta_{kn})(u_lu_n-\frac13 \delta_{ln})
	\\
	=&   s_+^2\(u_ku_lu_nu_n-\frac13 \delta_{kn}u_lu_n-\frac13 \delta_{ln}u_ku_n+\frac19 \delta_{kn}\delta_{ln}\)
	\\
	=&s^2\(\frac13u_ku_l+\frac19 \delta_{kl}\) =\frac {s_+}3 s_+(u_ku_l-\frac 13 \delta_{lk})+\frac {2s_+^2}9 \delta_{kl}
	\\
	=&\frac {s_+}3 Q_{kl}+\frac {2s_+^2}9 \delta_{kl}.
	\end{align*}
	Through the identity \eqref{ID}, we obtain
	\begin{align*}
	&Q_{lk}\frac{\p Q_{ij}}{\p x_l}\frac{\p Q_{ij}}{\p x_k}=
	(\frac{3}{s_+}Q_{ln}Q_{kn}-\frac {2s_+}3 \delta_{kl})\frac{\p Q_{ij}}{\p x_l}\frac{\p Q_{ij}}{\p x_k}\\
=&\frac{3}{s_+}(Q_{ln}\frac{\p Q_{ij}}{\p x_l})(Q_{kn}\frac{\p Q_{ij}}{\p x_k})-\frac {2s_+}3|\na Q|^2.
	\end{align*}
	\end{proof}
Recall from Longa  et al. \cite{LMT} that
	\begin{align*}
	L^{(4)}_5:=&Q_{\alpha\rho}Q_{\rho\beta}\frac{\p Q_{\alpha\mu}}{\p x_ \beta}\frac{\p Q_{\mu\nu}}{\p x_\nu},
	\quad
	L^{(4)}_6:=Q_{\alpha\rho}Q_{\rho\beta} \frac{\p Q_{\alpha\mu}}{\p  x_ \mu}  \frac{\p Q_{\beta\nu}}{\p x_ \nu },
	\\
	L^{(4)}_7:=&Q_{\alpha\rho}Q_{\rho\beta}  \frac{\p Q_{\alpha\mu}}{\p x_ \nu}\frac{\p  Q_{\beta\mu}}{\p x_\nu}.
	\end{align*}
Then we have
 \begin{lemma} \label{Lemma 2.1.b} For a uniaxial $Q\in S_*$, we obtain
\begin{align*}\alabel{Q}
	& Q_{ln}Q_{kn}\frac{\p Q_{ij}}{\p x_l}\frac{\p Q_{ij}}{\p x_k}=\frac{8}{5} L^{(4)}_5- \frac{2}{5}L^{(4)}_6 +\frac{2}{5}L_7^{(4)}.
	\end{align*}
\end{lemma}
\begin{proof} Let $Q= s_+(u\otimes u-\frac13 I)$ for $u\in S^2$.
Noting that  $u_i\na u_i =0$,  we calculate
\begin{align*}
	(u\times \curl u)_1^2=&[u_2(\na_1u_2-\na_2u_1)-u_3(\na_3u_1-\na_1u_3)]^2
	\\
	=&(-u_1\na_1u_1-u_2\na_2u_1-u_3\na_3u_1)^2=[(u\cdot \na)u_1]^2.
\end{align*}
Similarly, we can calculate other terms to obtain
\begin{align*}
	\sum_i[(u\cdot \na)u_i]^2=\sum_i(u\times \curl u)_i^2=|u\times \curl u|^2.
\end{align*}
Moreover, we calculate
	\begin{align*}\alabel{Q 3}
	& Q_{lk}\frac{\p Q_{ij}}{\p x_l}\frac{\p Q_{ij}}{\p x_k}=s_+^{3}  (u_lu_k-\frac 13\delta_{lk})\na_l(u_iu_j)\na_k(u_iu_j)
	\\
	=& s_+^{3} (u_lu_k-\frac 13\delta_{lk})(u_j\na_lu_i+u_i\na_lu_j)(u_j\na_ku_i+u_i\na_ku_j)
	\\
	=& s_+^{3} (u_lu_k-\frac 13\delta_{lk})(\na_lu_i\na_ku_i+\na_lu_j\na_ku_j)
	\\
	=&2s_+^{3} \sum_i[( u\cdot \na) u_i]^2-\frac 23s_+^{3} |\na u|^2=2s_+^{3} |u\times \curl u|^2-\frac 23s_+^{3} |\na u|^2.
	\end{align*}
It follows from using \eqref{Q third} and \eqref{Q 3}   that
	\begin{align}\label{Q4}
 Q_{ln}Q_{kn}\frac{\p Q_{ij}}{\p x_l}\frac{\p Q_{ij}}{\p x_k}&=\frac 1{3}s_+ Q_{lk}\frac{\p Q_{ij}}{\p x_l}\frac{\p Q_{ij}}{\p x_k}+  \frac {2s^2_+}9|\na Q|^2\\
	&=\frac 2{3}s_+^{4}|u\times \curl u|^2+\frac {2}9s_+^4|\na u|^2.\nonumber
	\end{align}
We can verify from    \cite{LMT} that
\begin{align*}\alabel{Longa}
4L_5^{(4)}-L_6^{(4)}=\frac {5s_+^4}3|u\times \curl u|^2,\quad L_7^{(4)}=\frac 59s_+^4|\na u|^2.
	\end{align*}
Substituting  \eqref{Longa}  into \eqref{Q4}, we have
	\[Q_{ln}Q_{kn}\frac{\p Q_{ij}}{\p x_l}\frac{\p Q_{ij}}{\p x_k}=\frac{2}{5}(4L^{(4)}_5-L^{(4)}_6)+\frac{2}{5}L_7^{(4)}.\]
	\end{proof}
Under the condition \eqref{L},
one can verify from Lemma 1.2 in  \cite{KRSZ}   that there are two uniform constants $\alpha>0$ and $C>0$ such that the form $f_{E,1}(Q,p)$ also satisfies
\begin{align}\label{CV}
	\frac {\alpha}2 |p|^2&\leq   f_{E,1}(Q,p)\leq C(1+|Q|^2)|p|^2, \quad|\p_Qf_{E,1}(Q,p)|\leq C(1+|Q|)|p|^2
\end{align} for any  $Q\in \mathbb{M}^{3\times 3}$ and   $p \in\mathbb{M}^{3\times 3}\times \R^3$.
Since   $ f_{E,1}(Q,p)$ is
quadratic in $p$ and satisfies \eqref{CV}, it can be checked  (c.f. \cite{HM})  that $  f_{E,1}(Q,p)$ is uniformly convex in $p$; that is
\begin{align}\label{sec2 f_E}
	\p^2_{p^i_{kl}p^j_{mn}} f_{E,1}(Q,p)\xi^i_{kl}\xi^j_{mn}\geq  \frac {\alpha}2 |\xi|^2,  \quad \forall \xi\in\mathbb{M}^{3\times 3}\times \R^3.
\end{align}
	Now we give a proof of Theorem 1.
\begin{proof} Under the condition on $L_1, \cdots, L_4$ in Theorem \ref {Theorem 1},    it is clear from using  Lemma 1.2  in \cite{KRSZ}  that \[f_{E,1}(Q,\nabla Q)\geq   \frac {\alpha}2 |\na Q|^2, \quad \forall Q\in S_0.\]
By the standard theory in the calculus  of  variations (c.f. \cite{Ga}), there exists a  minimizer $Q_{L}$  of $E_L$ in $W_{Q_0}^{1,2}(\Omega ; S_0)$.
For each    $Q\in W^{1,2}_{Q_0}(\Omega ; S_0)$, we set
\[E(Q ;\Omega ):=\int_{\Omega} f_{E,1}(Q,\nabla Q)\,dx.\]
It implies that
\[E(Q_L ;\Omega )+ \int_{\Omega}(f_B(Q_L)-  \inf_{S_0} f_B)\,dx  \leq  E(Q ;\Omega )
\]
for any  $Q\in W^{1,2}_{Q_0}(\Omega ; S_*)$ with the fact that $\tilde f_B(Q)=f_B(Q)-  \inf_{S_0} f_B=0$.

As $L\to 0$, minimizers  $Q_{L}$  converge  (possible passing
subsequence)  weakly to a tensor $Q_*\in W^{1,2}(\Omega ; S_0)$ with that $\tilde f_B( Q_*)=0$, which implies that $Q_*\in S_*$ a.e. in $\Omega$. Then, for any  $Q\in W_{Q_0}^{1,2}(\Omega ; S_*)$, we have
$$E (Q_* ;\Omega )\leq \liminf_{L\to 0}E(Q_{L} ;\Omega ) \leq \limsup_{L\to 0}E(Q_{L} ;\Omega )\leq E(Q;\Omega).
$$
  Therefore $Q_*$ is also a minimizer of $E$ in $W^{1,2}_{Q_0}(\Omega ; S_*)$.
Choosing $Q=Q_*$ in the above inequality, it implies that
\[E (Q_* ;\Omega )=\lim_{L\to 0}E_L(Q_{L} ;\Omega ),\quad \lim_{L\to 0}\frac 1 L\int_{\Omega} \tilde f_B(Q_{L}) \,dx=0.\]
 Moreover, it is known that
\begin{align*}
 \int_{\Omega} |\nabla Q_*|^2\,dx\leq& \liminf_{L\to 0}\int_{\Omega} |\nabla Q_L|^2\,dx ,\\
  \int_{\Omega} V(Q_*, \nabla Q_*)\,dx\leq& \liminf_{L\to 0}\int_{\Omega} V(Q_{L},\nabla Q_L )\,dx.
\end{align*}

It implies that $\int_{\Omega} |\nabla Q_*|^2\,dx= \liminf_{L\to 0}\int_{\Omega} |\nabla Q_L|^2\,dx$. Otherwise, there is a subsequence $L_k\to 0$ such that
\[\int_{\Omega} |\nabla Q_*|^2\,dx< \lim_{L_k\to 0}\int_{\Omega} |\nabla Q_{L_k}|^2\,dx.\]
Then
\begin{align*}
		&
E (Q_* ;\Omega )=\lim_{L_k\to 0}E_{L_k}(Q_{L_k} ;\Omega ),\\
=& \frac {\alpha}2 \lim_{L_k\to 0}\int_{\Omega} |\nabla Q_{L_k}|^2\,dx+ \lim_{L_k\to 0}\int_{\Omega} V(Q_{L_k}, \nabla Q_{L_k})\,dx\\
< &E (Q_* ;\Omega ).
\end{align*}
 This is impossible. Therefore,  minimizers $Q_{L_k}$ converge strongly, up to a subsequence, to a minimizer $Q_*=s_+(u_*\otimes u_*-\frac 1 3 I)$ of $E$ in
$W^{1,2}_{Q_0}(\Omega ; S_0)$. Following from Lemma \ref{Lemma 2.2}, $Q_*$ satisfies \eqref{EL} and $u_*$ is a minimizer of the Oseen-Frank energy in $W^{1,2}(\Omega; S^2)$. Due to the well-known result of Hardt, Kinderlehrer and Lin \cite{HKL}, $u_*$ is partially regular in $\Omega$ (see also \cite{Ho1}). Thus $Q_*$ is partially regular.
\end{proof}
	
\begin{lemma}
	If $Q$ is a minimizer of $E_{L}(Q;\Omega)$ from \eqref{LG1} in $W^{1,2}_{Q_0}(\Omega; S_0)$, it   satisfies
	\begin{align*}
	&- \alpha  \Delta Q_{ij}-\frac 1 2  \na_k (V_{p^k_{ij}}+ V_{p^k_{ji}})+\frac 1 3 \delta_{ij}\sum_{l=1}^3 \nabla_k V_{p^k_{ll}}  +\frac 1 2 (V_{Q_{ij}}  +V_{Q_{ji}})-   \frac 1 3 \delta_{ij}\sum_{l=1}^3 V_{Q_{ll}}  \\& + \frac 1 L \(-aQ_{ij}- b  (Q_{ik}Q_{kj}-\frac 13\delta_{ij}\tr(Q^2) )+cQ_{ij}\tr(Q^2)\)=0
	\end{align*}
	in the weak sense, where 	$V_{p^k_{ij}}:=V_{p^k_{ij}}(Q,p)$ with $p=(\nabla_k Q_{ij})$.
\end{lemma}

\begin{proof}
For any test function $\phi\in C^\infty_0(\Omega; S_0)$, consider $Q_t:=Q+t\phi$ for $t\in \R$.
Then for all $\phi\in C^\infty_0(\Omega; S_0)$, we calculate
\begin{align*}
&\left.\int_\Omega \frac{d}{d t}\(\tilde f_{E,1}(Q_t,\na Q_t)+\frac 1L\tilde f_B(Q_t)\)\right|_{t=0}\d x
\\
=&\int_\Omega \alpha  \frac{\p Q_{ij}}{\p x_k}\frac{\p \phi_{ij}}{\p x_k}+ V_{p^k_{ij}} \frac{\p \phi_{ij}}{\p x_k}+V_{Q_{ij}} \phi_{ij}\d x
\\
&+ \frac 1 L\int_\Omega  -aQ_{ij}\phi_{ij}-  b  Q_{ik}Q_{kj}\phi_{ij} +c(Q_{ij}\tr(Q^2)\phi_{ij})\d x
\\
=&\int_\Omega\( -\alpha  \Delta Q_{ij} -\frac 1 2\frac{\p }{\p x_k}(V_{p^k_{ij}} + V_{p^k_{ji}})
+\frac 12 (V_{Q_{ij}} +V_{Q_{ji}}) \)\phi_{ij}\d x \\
&+ \frac 1 L\int_\Omega  \(-aQ_{ij}- b  Q_{ik}Q_{kj} +cQ_{ij}\tr(Q^2)\)\phi_{ij}\d x\\
=& \int_\Omega\( -\alpha  \Delta Q_{ij} -\frac 1 2 \nabla_k(V_{p^k_{ij}}+  V_{p^k_{ji}}) -\frac 1 3 \delta_{ij}\sum_l \nabla_k V_{p^k_{ll}}
  \)\phi_{ij}\d x
\\
& +\int_\Omega\(\frac 12 (V_{Q_{ij}} +V_{Q_{ji}}) - \frac 1 3 \delta_{ij}\sum_l V_{Q_{ll}}  \)\phi_{ij}\d x
\\
& + \frac 1 L\int_\Omega  \(-aQ_{ij}- b  (Q_{ik}Q_{kj}-\frac 1 3 \delta_{ij}\tr(Q^2) ) +cQ_{ij}\tr(Q^2)\)\phi_{ij}\d x
=  0,
\end{align*}
where we used the fact that $\phi$ is traceless.
 This proves our claim.
\end{proof}

In the case of $L_2=L_3=L_4=0$,  Majumdar and Zarnescu  \cite {MZ}  proved that the weak solution of \eqref{MDEL} is  bounded by using a maximum principle. However, when $L_2$, $L_3$, $L_4$ are non-zero, the  system \eqref{MDEL} is a nonlinear elliptic system, so there exists no such maximum principle for it (e.g. \cite {Ga}, \cite{Giu}). Therefore, it is not clear   whether each minimizer $Q_L$ of $E_{L}(Q; \Omega )$   in $W_{Q_0}^{1,2}(\Omega, S_0)$ is bounded and the energy density $f_{E,1}(Q, \nabla Q)$ in (\ref {LG1})   can be bounded above by $C|\nabla Q|^2 +C$. Without this above growth condition on the density,  it is  a well-known fact that a minimizer $Q_L$ of  the   Landau-de Gennes  energy functional in $W^{1,2}_{Q_0}(\Omega; S_0 )$ may not satisfy the Euler-Lagrange equation in $W^{1,2}(\Omega, S_0)$. To overcome this difficulty, we can introduce a smooth cutoff function $\eta (r)$ in $[0, \infty )$ so that $\eta (r)=1$ for $r\leq M$ with a very large constant $M>0$ and $\eta (r)=0$ for $r\geq M+1$.
Then for each $Q\in W^{1,2}(\Omega, S_0)$, one can modify the Landau-de Gennes  density by
\begin{equation}\label{MD}
\widetilde f_E(Q, \nabla Q):= \frac {\alpha}2 |\nabla Q|^2 +\tilde V(Q, \nabla Q)=  \frac {\alpha}2|\nabla Q|^2 +\eta (|Q|) V(Q, \nabla Q)
 \end{equation}
with the property that
\[ \frac {\alpha}2  |\nabla Q|^2\leq \widetilde f_E(Q, \nabla Q)\leq C|\nabla Q|^2. \]
For a large $M>0$ in (\ref{MD}),  we consider a modified   Landau-de Gennes functional
\begin{equation}\label{MDLG}
\tilde E_{L}(Q; \Omega )=\int_{\Omega}\( \widetilde f_E(Q, \nabla Q) +\frac 1 L \tilde f_B(Q)\)\,dx.
\end{equation}
Then we obtain
\begin{lemma}\label{lem uniform bound}
Let $Q_L$ be a weak solution to the equation (\ref{MDEL}) with the boundary value $Q_0\in W^{1,2}(\Omega ; S_*)$ associated to  the functional $\tilde f_E(Q,\D Q)$ in \eqref{MD}. Then,  $|Q_L|\leq M+1$ for a sufficient large $M$.
\end{lemma}
\begin{proof} Recall from the definition of $\tilde f_E(Q,\D Q)$ in (\ref{MD}) that for a $Q\in S_0$ with $|Q|\geq M+1$,
\[\tilde f_E(Q, \nabla Q)=\frac {\alpha} 2|\nabla Q|^2.\]
Similarly to one in \cite {CHH}, choose  a test function $\phi =Q (1-\min\{1, \frac {M+1}{|Q|}\})$. Multiplying (\ref{MDEL})   by the test function $\phi$, we have
\begin{align*}
	&\alpha  \int_{|Q|\geq M+1}|\nabla Q|^2 (1- \frac {M+1}{|Q|})-(M+1) Q_{ij} \na_k Q_{ij} \nabla_k  \frac 1{|Q|})\,dx
 \\& + \frac 1 L \int_{|Q|\geq M+1} \(-a|Q|^2- b  Q_{ik}Q_{kj} Q_{ij} +c|Q|^4\)
 (1- \frac {M+1}{|Q|}) \,dx=0.
	\end{align*}
Note the fact that $  \nabla_k |Q|^2=2Q_{ij} \na_k Q_{ij} $. The above second term is non-negative.
For a sufficiently large $M>0$, the third term is positive. This implies that the set $\{|Q|\geq M+1\}$ is empty; i.e.,
$|Q|\leq M+1$ a.e. in $\Omega$.
\end{proof}

The following result  is a variant result of Giaquinta-Giusti \cite{GG}  (see more details in page 206 of \cite{Gi}):
\begin{prop}\label{{Pro}}  For each $L>0$, let  $Q_L$  be a  bounded minimizer of \eqref{MDLG}
		 in $W^{1,2}_{Q_L}(\Omega; S_0 )$. Then there exists an open set $\Omega_L\subset \Omega$ such that $Q_L\in C_{loc}^{1, \alpha }(\Omega \backslash \Omega_L)$ for each $\alpha<1$. Moreover, there is a small constant $\varepsilon_0$ independent of $Q_L$ such that
\[ \Sigma_L:=\Omega\backslash \Omega_L=\left \{x_0\in \Omega:  \, \liminf_{R\to 0}\int_{B_r(x_0)}|\nabla Q_L|^2\,dx >\varepsilon_0\right \}\]
and the Hausdorff measure $\mathcal H^{q}(\Sigma_L )=0$ with $0<q<1$.
\end{prop}

\section{Proof of Theorem  \ref{Theorem 2}}
 At first, let us  recall that for a  uniaxial tensor    $Q\in W^{1,2}_{Q_0}(\Omega; S_* )$,  its energy is given by
\[E(Q; \Omega):=\int_{\Omega}   f_{E}(Q, \nabla Q)\,dx,\]
where $f_{E}(Q, \nabla Q)=\frac {\alpha}2 |\D Q|^2+V(Q, \D Q)$. Then we have
\begin{lemma}\label{Lemma 2.2} If $Q$ is a minimizer of $E(Q; \Omega)$ in $W^{1,2}_{Q_0}(\Omega; S_*)$, it   satisfies
\begin{align*}
		&
		\quad   \alpha \(- s_+\Delta Q_{ij} +2 \na_kQ_{il}\na_kQ_{jl} -2 s_+^{-1}(Q_{ij}+\frac {s_+}3\delta_{ij})|\na Q|^2\)\\
		&-\nabla_k \(
	(Q_{jl}+\frac {s_+}3 \delta_{jl})V_{p^k_{il}}
	+(Q_{il}+\frac {s_+}3 \delta_{il})V_{p^k_{jl}}
  -2s_+^{-1} (Q_{ij}+\frac {s_+}3 \delta_{ij})(Q_{lm}+\frac {s_+}3 \delta_{lm})V_{p^k_{lm}} \)
		\\
		& + V_{p^k_{il}}\na_kQ_{jl}
	+ V_{p^k_{jl}}\na_k Q_{il}
  -2s_+^{-1}V_{p^k_{lm}}\(\nabla_k Q_{lm} ( Q_{ij}+\frac {s_+}3 \delta_{ij})+( Q_{lm}+\frac {s_+}3 \delta_{lm})\nabla_k Q_{ij}\)
		\\
&+V_{Q_{il}}(Q_{jl}+\frac {s_+}3 \delta_{jl})
	+V_{Q_{jl}}(Q_{il}+\frac {s_+}3 \delta_{il})
  -2s_+^{-1}V_{Q_{lm}}(Q_{lm}+\frac {s_+}3 \delta_{lm}) (Q_{ij}+\frac {s_+}3 \delta_{ij})
\\
		& =0
	\end{align*}
	in the weak sense.
	\end{lemma}
\begin{proof} Let $\phi\in C^\infty_0(\Omega; \R^3)$ be  a test function. For each $u_t = \frac{u+t\phi}{|u+t\phi|}$ with $t\in \R$,
		we define
		\begin{align}\label{eq Q variation}
		Q_t(x):=Q(u_t(x))  = s_+\(u_t(x)\otimes u_t(x)-\frac 1 3 I\)\in S_*.
		\end{align}
 For any $\eta\in C^\infty_0(\Omega; S_0)$, we choose  a test function $\phi$ such that $\phi_{i}:=u_k\eta_{ik} $.
If $Q$ is a minimizer, the    first variation of the energy of $Q$ is zero; that is
	\begin{align*}
	\left.\frac{\d }{\d t}\int_\Omega f_E(Q_t,\nabla Q_t)\d x\right|_{t=0}=\int_\Omega\left. f_{Q_{ij}}\frac{\d Q_{t; ij}}{\d t} +f_{p_{ij}^k}\frac{\d }{\d t}\frac{\p  Q_{ij}}{\p x^k}\d x \right|_{t=0}=0.
	\end{align*}
Note that
\begin{align*}
		   \frac{\d  Q_{t; ij}}{\d t} =
		   &\frac{\Big(
		   	Q_{jl}+\frac {s_+}3 \delta_{jl}
		   	+t(Q_{lm}+\frac {s_+}3 \delta_{lm})\Big)\eta_{il}
		   	+\Big(Q_{il}+\frac {s_+}3 \delta_{il}+t(Q_{lm}+\frac {s_+}3 \delta_{lm})\Big)\eta_{jl}}
	   	{1+2ts_+^{-1}Q_{il}\eta_{il}+t^2s_+^{-1}(Q_{lm}+\frac {s_+}3 \delta_{lm})\eta_{il}\eta_{im}}
		   \\
		   &-\frac{2s_+^{-1}\Big(Q_{ij}+\frac {s_+}3 \delta_{ij}+t(Q_{lm}+\frac {s_+}3 \delta_{lm})\eta_{il}\eta_{im}\Big)
		   	(Q_{lm}+\frac {s_+}3 \delta_{lm})\eta_{lm}}
		   {|1+2ts_+^{-1}Q_{il}\eta_{il}+t^2s_+^{-1}(Q_{lm}+\frac {s_+}3 \delta_{lm})\eta_{il}\eta_{im}|^2},\nonumber
		\end{align*}
where we used  the fact that $|u|=1$ and $\phi_{i}=u_l\eta_{il} $. Then we have
\begin{align}\label{V1}
&\left.\frac{\d  Q_{t; i,j}}{\d t} \right|_{t=0}=s_+\big( u_j\phi_i+u_i\phi_j  -2(u\cdot \phi)(u_i u_j)\big)
\\
=&( Q_{jl}+\frac {s_+}3 \delta_{jl})\eta_{il}+(Q_{il}+\frac {s_+}3 \delta_{il})\eta_{jl}
-2s_+^{-1}(Q_{ij}+\frac {s_+}3 \delta_{ij})(Q_{lm}+\frac {s_+}3 \delta_{lm})\eta_{lm}.\nonumber
\end{align}
Using the fact that $\nabla_k |u+t\phi |^2 =0$ at $t=0$  and  substituting $\phi_{i}:=u_l\eta_{il} $, a simple calculation shows
\begin{align}\label{V2}
		   &  \left.\frac{\d }{\d t}\frac{\p   Q_{t; ij}}{\p x_k}\right|_{t=0}=\left . \( \frac{\p } {\p x_k}\frac{\d }{\d t}    Q_{t; ij}\)\right|_{t=0}
\\
&=\frac{\p Q_{jl}}{\p x_k} \eta_{il}+\frac{\p Q_{il}}{\p x_k}\eta_{jl}
  -2s_+^{-1}\(\frac{\p Q_{ij}}{\p x_k}Q_{lm}+\frac{\p Q_{lm}}{\p x_k}(Q_{ij}+\frac {s_+}3 \delta_{ij})\)\eta_{lm}
           \nonumber \\
            &+(Q_{jl}+\frac  {s_+}3 \delta_{jl})\frac{\p \eta_{il}}{\p x_k}+(Q_{il}+\frac {s_+}3 \delta_{il})\frac{\p \eta_{jl}}{\p x_k}
  -2s_+^{-1}( Q_{ij}+\frac {s_+}3 \delta_{ij})(Q_{lm}+\frac {s_+}3 \delta_{lm})\frac{\p \eta_{lm}}{\p x_k}.\nonumber
		    \end{align}

In the special case of $\frac 12 \int_{\Omega} {|\na Q|^2}\d x$, it follows from using (\ref{V2})  and   $\<Q, \nabla Q\>=0$ that
		\begin{align*}
		& \frac{d}{d t}\int_{\Omega}\left.\frac{|\na Q_t|^2}{2}\d x\right|_{t=0}=\int_\Omega  \nabla_k Q_{t; ij} \left. \frac { d \nabla_k Q_{t; ij}} {dt}  \right|_{t=0}\d x
		\\
		=&\int_{\Omega}2 \na_kQ_{il}\na_kQ_{jl}\eta_{ij}-2 (s_+^{-1}Q_{ij}+\frac 13\delta_{ij})|\na Q|^2 \eta_{ij}\d x
		\\
		&+\frac12s_+ \int_{\Omega}(\nabla_k Q_{il} \nabla_k\eta_{il} + \nabla_k Q_{lj}\nabla_k\eta_{jl})\d x\\
		=&\int_{\Omega}\(-s_+ \Delta Q_{ij} +2 \na_kQ_{il}\na_kQ_{jl} -2 (s_+^{-1}Q_{ij}+\frac 13\delta_{ij})|\na Q|^2\) \eta_{ij}\d x\alabel{EL for one constant on S_*}
		\end{align*}
		for all $\eta\in C^\infty_0(\Omega; S_0)$.

	For the term $V(Q,\na Q)$, using  (\ref{V1})-(\ref{V2})  and integrating by parts, we have
	\begin{align*}
	&\int_\Omega\left.\frac{\d }{\d t}V(Q_t,\nabla Q_t)\right|_{t=0}\d x=\int_\Omega \left. [V_{p^k_{ij}} \frac { d \nabla_k Q_{t; ij}} {dt}+V_{Q_{ij}}\frac {d  Q_{t; ij}}{dt} ]\right|_{t=0}\d x
	\alabel{EL for V on S_*}
	\\
	=&\int_\Omega V_{p^k_{ij}}\((Q_{jl}+\frac {s_+}3 \delta_{jl})\frac{\p \eta_{il}}{\p x_k}+(Q_{il}+\frac {s_+}3 \delta_{il})\frac{\p \eta_{jl}}{\p x_k}+\frac{\p Q_{jl}}{\p x_k} \eta_{il}+\frac{\p Q_{il}}{\p x_k}\eta_{jl}\)\d x
  \\
  		&-2s_+^{-1}\int_\Omega
  	  V_{p^k_{ij}}\(\frac{\p Q_{ij}}{\p x_k} (Q_{lm}+\frac {s_+}3 \delta_{lm})+\frac{\p Q_{lm}}{\p x_k}(Q_{ij}+\frac {s_+}3 \delta_{ij})\)\eta_{lm}\d x
  	  \\
  &+\int_\Omega -2s_+^{-1}V_{p^k_{ij}}(Q_{ij}+\frac {s_+}3 \delta_{ij}) (Q_{lm}+\frac {s_+}3 \delta_{lm})\frac{\p \eta_{lm}}{\p x_k}+V_{Q_{ij}} ( Q_{jl}+\frac {s_+}3 \delta_{jl})\eta_{il}\d x
	\\
	&+\int_\Omega V_{Q_{ij}}\((Q_{il}+\frac {s_+}3 \delta_{il})\eta_{jl}
  -2s_+^{-1}(Q_{ij}+\frac {s_+}3 \delta_{ij}) (Q_{lm}+\frac {s_+}3 \delta_{lm})\eta_{lm} \) \d x
	\\
	=&-\int_\Omega \frac{\p }{\p x_k}\(
	(Q_{jl}+\frac {s_+}3 \delta_{jl})V_{p^k_{il}}
	+(Q_{il}+\frac {s_+}3 \delta_{il})V_{p^k_{jl}}\)\eta_{ij}
  \d x
  \\
  &+\int_\Omega \frac{\p }{\p x_k}\(2s_+^{-1} (Q_{ij}+\frac {s_+}3 \delta_{ij})(Q_{lm}+\frac {s_+}3 \delta_{lm})V_{p^k_{lm}} \)\eta_{ij}+V_{p^k_{il}}\frac{\p Q_{jl}}{\p x_k}\eta_{ij}\d x
	\\
	&+\int_\Omega\(
	  	V_{p^k_{jl}}\frac{\p Q_{il}}{\p x_k}
  -2s_+^{-1}V_{p^k_{lm}}\(\frac{\p Q_{lm}}{\p x_k}(Q_{ij}+\frac {s_+}3 \delta_{ij})+( Q_{lm}+\frac {s_+}3 \delta_{lm})\frac{\p Q_{ij}}{\p x_k}\)\)\eta_{ij}\d x
  \\
	&+\int_\Omega \( V_{Q_{il}}(Q_{jl}+\frac {s_+}3 \delta_{jl})
	+V_{Q_{jl}}(Q_{il}+\frac {s_+}3 \delta_{il})\)\eta_{ij}\d x
  \\
  &-2s_+^{-1}\int_\Omega
    V_{Q_{lm}}(Q_{lm}+\frac {s_+}3 \delta_{lm}) (Q_{ij}+\frac {s_+}3 \delta_{ij})\eta_{ij}\d x.
	\end{align*}
	 Combining  above two identities \eqref{EL for one constant on S_*} and \eqref{EL for V on S_*}, we prove  Lemma  \ref{Lemma 2.2}.
\end{proof}

\begin{cor} Assume that $Q=s_+(u\otimes u-\frac 1 3 I)$. Then $Q=(Q_{ij})$ is a solution of equation
\begin{equation}\label{QH}
\Delta Q_{ij}-2s_+^{-1}\na_kQ_{il}\na_kQ_{jl}+2s_+^{-1}(s_+^{-1}Q_{ij}+\frac 13\delta_{ij})|\na Q|^2=0
\end{equation}
 if and only if $u$ is a harmonic map from $\Omega$ into $S^2$; i.e., $-\Delta u=|\nabla u|^2 u$.
	 \end{cor}

Now we give a proof of Theorem \ref{Theorem 2}.
\begin{proof}For each $L>0$, let $Q_L$ be a weak solution to the equation (\ref{MDEL}) with boundary value $Q_0\in W^{1,2}(\Omega, S_*)$
		and assume that $Q_L$ is uniformly bounded in $\Omega$.

For each $\delta>0$, define a set
		\[\Sigma_{\delta}=S_0\backslash S_{\delta}=\{Q\in S_0: \mbox{dist}(Q, S_*)\geq \delta\}.\]
For each $Q\in \Sigma_{\delta}$, we have $\pi (Q)\in S_*$;
i.e., $ \pi (Q) =s_+\(u\otimes u-\frac 1 3 I\)$ with $u\in S^2$.

For a   $ \pi (Q_L) =s_+\(u_L\otimes u_L-\frac 1 3 I\)$ with $u_L\in S^2$ and a  test function $\phi\in C^\infty_0(\Omega; \R^3)$ with a small $t\in\R$, we set  $u_{L, t}:= \frac{u_L+t\phi}{|u_L+t\phi|}$.
  Then we have
	\begin{align}\label{eq projection variation}
	 \big(\pi (Q_L)\big)_t :=  \ s_+\(u_{L, t}\otimes u_{L, t}-\frac 1 3 I\)\in S_*.
		\end{align}
For any $Q\in S_{\delta}$, set
\[  F(Q):=\tilde f_B(Q)+|Q-\pi(Q)|^2.\]
Using the Taylor expansion of $F\((\pi (Q_L))_t\)$ at $Q_L\in S_{\delta}$,  we derive
	\begin{align*}\alabel{eq Taylor}
	&\frac{F\((\pi (Q_L))_t\)}L= \frac{F(Q_L)}{L}
	+\frac{(\nabla F(Q_L) )_{ij}}{L} \((\pi (Q_L))_t -Q_L\)_{ij}
	\\
	&\qquad\qquad+\frac1{2L}\na^2_{Q_{ij}Q_{kl}}F\(Q_{\tau_1}\)\((\pi (Q_L))_t-Q_L\)_{ij}\( (\pi (Q_L))_t-Q_L\)_{kl},
	\end{align*}
where $Q_{\tau_1}:=(1-\tau_1) (\pi (Q_L))_t+\tau_1 Q_L$ for some $\tau_1\in[0,1]$. Note that
\[|Q_{\tau_1}-(\pi (Q_L))_t|\leq |Q_L-\pi (Q_L)|+ |\pi (Q_L)-(\pi (Q_L))_t|.\]
Since  $(\pi (Q_L))_t\in S_*$, it implies that $F((\pi (Q_L))_t)=0$.  Note that the function $F(Q)$ is smooth in $Q$.
For sufficiently small $t$, we have $|Q_{\tau_1}-(\pi (Q_L))_t|\leq 3\delta$ for $Q_L\in S_{2\delta}$.
For each $Q\in S^*$,  it is known that the Hessian of $\tilde f_B(Q)$ is semi-positive definite at $Q=Q^*$. Therefore  each $Q\in S_{\delta}$, the  Hessian of $F_B(Q)$  is positive definite  with sufficiently small $\delta >0$; i.e., for any $Q_L\in S_{2\delta}$,
 we have
\begin{align*}\alabel{qu}
&\na^2_{Q_{ij}Q_{kl}} F\big (Q_\tau\big )( (\pi (Q_L))_t-Q_L)_{ij}( (\pi (Q_L))_t-Q_L)_{kl}\geq \frac 1 2  |(\pi (Q_L))_t-Q_L|^2
\end{align*}
with sufficiently small $t$ and $\delta$.
Then it follows from \eqref{eq Taylor}-\eqref{qu} that
	\begin{align*}\alabel{eq f_B 2 delta}
	&\int_{\Omega_{L,2\delta}} \frac{(\nabla F(Q_L) )_{ij}}{L} ((\pi (Q_L))_t -Q_L)_{ij}\,dx\\
	=&\int_{\Omega_{L,2\delta}} \frac{\na_{Q_{ij}} f_B(Q_L)}L((\pi (Q_L))_t -Q_L)_{ij}\d x
	\\
	&+2\int_{\Omega_{L,2\delta}}\p_{Q_{ij}}(Q_L-\pi(Q_L))_{mn}\frac{(Q_L-\pi(Q_L))_{mn}}{L} ((\pi (Q_L))_t -Q_L)_{ij}\d x
	\\
	\leq&-\frac 12\int_{\Omega_{L,2\delta}}\frac{|(\pi (Q_L))_t-Q_L|^2}{L}\d x
	\end{align*}
provided  $\Omega_{L,2\delta}=\{x\in  \Omega: Q_L(x)\in S_{2\delta}\}$ for $\delta >0$.

By using   Young's inequality, we have
	\begin{align*}\alabel{eq f_B 2 delta p2}
	&\int_{\Omega_{L,2\delta}} \frac1L\na_{Q_{ij}} f_B(Q_L)((\pi (Q_L))_t -Q_L)_{ij} \d x +\frac 14\int_{\Omega_{L,2\delta}}\frac{|(\pi (Q_L))_t-Q_L|^2}{L}\d x
	\\
	\leq& C\int_{\Omega_{L,2\delta}}\frac{|Q_L-\pi(Q_L)|^2}{L} \d x\leq C \int_{\Omega_{L,2\delta}} \frac1L \tilde f_B(Q_L)\,dx,
	\end{align*}
where we used a result in \cite{NZ} or  Corollary 2 in Section 3.

 In order to  extend \eqref{eq f_B 2 delta p2} to $\Omega$, we  define
\begin{align*}\alabel{eq extension}
\hat  Q_{L,t}:=\begin{cases}
\pi(Q_L)_t,&\mbox{ for }Q_L\in S_{\delta}
\\
\frac{|Q_L-\pi(Q_L)|^2}{\delta^2}\pi(Q_L)_t+\frac{\delta^2-|Q_L-\pi(Q_L)|^2}{\delta^2} Q_{*,t},&\mbox{ for }Q_L\in \Sigma_{\delta}\backslash \Sigma_{2\delta}
\\
Q_{*,t},&\mbox{ for } Q_L\in \Sigma_{2\delta}.
\end{cases}
\end{align*}
It can be checked that $\hat  Q_{L,t}\in W^{1,2}_{Q_0}(\Omega ; S_0)$. Then
\begin{align*}\alabel{TB}
\hat  Q_{L,t}-Q_{*,t}=\begin{cases}
\pi(Q_L)_t- Q_{*,t},&\mbox{ for }Q_L\in S_{\delta}
\\
\frac{|Q_L-\pi(Q_L)|^2}{\delta^2} (\pi(Q_L)_t- Q_{*,t}),&\mbox{ for }  Q_L\in \Sigma_{\delta}\backslash \Sigma_{2\delta}
\\
0,&\mbox{ for } Q_L\in \Sigma_{2\delta}.
\end{cases}
\end{align*}

On the other hand, there exists a uniform bound $C(\delta)>0$ such that for all $x\in \Omega \backslash \Omega_{L,\delta}$,  $\tilde f_B(Q_L(x)) \geq C(\delta)$. Using Lemma \ref{lem uniform bound}, we observe that
\begin{align*}\alabel{eq f_B out delta}
&\int_{\Omega \backslash \Omega_{L,\delta} }\frac1L\na_{Q_{ij}} f_B(Q_L)(\hat Q_{L,t}-Q_L)_{ij}\d x
\\
=&\int_{\Omega_{L,2\delta} \backslash \Omega_{L,\delta} }\frac1L\na_{Q_{ij}} f_B(Q_L)\left [\frac{|Q_L-\pi(Q_L)|^2}{\delta^2}(\pi(Q_L)_t-Q_{*,t})+(Q_{*,t}-Q_L)\right]_{ij}\d x
\\
&+\int_{\Omega \backslash \Omega_{L,2\delta} }\frac1L\na_{Q_{ij}} f_B(Q_L)(Q_{*,t}-Q_L)_{ij}\d x \\
\leq& C \frac{| \Omega \backslash \Omega_{L,\delta}|}{L}  \leq  \frac C {C(\delta)}\int_{\Omega \backslash \Omega_{L,\delta} }\frac{\tilde f_B(Q_L) }L \d x.
\end{align*}
By the assumption \eqref{eq Theorem 3} in Theorem 2, we deduce from \eqref{eq f_B 2 delta p2} and \eqref{eq f_B out delta} that
\begin{align*}\alabel{eq na f_B}
\lim_{L\to 0}\int_{\Omega}\frac1L\na_{Q_{ij}} f_B(Q_L)(\hat Q_{L,t}-Q_L)_{ij}\d x\leq 0.
\end{align*}
Multiplying (\ref{MDEL}) by $(\hat Q_{L,t}-Q_L)$ and using \eqref{eq na f_B} yield
\begin{align*}\alabel{eq sub of na f_B}
&\lim_{L\to0}\int_{\Omega }\(\alpha \na_kQ_{L,ij} +\tilde V_{p^k_{ij}}(Q_L,\na Q_L)-\tilde V_{Q_{ij}}(Q_L,\na Q_L)\)\na_k (\hat Q_{L,t}-Q_L)_{ij}\d x\geq 0.
\end{align*}
Here we used the fact that $\hat Q_{L,t}-Q_L$ is symmetric and traceless.

In order to pass a limit, we claim that $\hat  Q_{L,t}\to Q_{*,t}$ strongly in $W^{1,2}_{Q_0}(\Omega ; S_0)$.

\noindent In fact,  it follows from \eqref{TB} that
\begin{align*}\alabel{Q hat convergence}
&\int_{\Omega} |\na  (\hat  Q_{L,t}-Q_{*,t})|^2 \,dx= \int_{ \Omega_{L,2\delta}}  |\na  (\hat  Q_{L,t}-Q_{*,t})|^2 \,dx
\\
= & \int_{\Omega_{L,\delta}} |\na  ( \hat Q_{L,t}-Q_{*,t})|^2  \,dx
+\int_{ \Omega_{L,2\delta} \backslash  \Omega_{L,\delta} } \left |\na \(  \frac{|Q_L-\pi(Q_L)|^2}{\delta^2} (\pi(Q_L)_t- Q_{*,t})\)\right |^2  \,dx
\\
\leq&  \int_{\Omega_{L,\delta}}  |\na  ( (\pi (Q_L))_t-\pi (Q_{*})_t)|^2  \,dx +
 C \int_{ \Omega_{L,2\delta} \backslash  \Omega_{L,\delta} }|\na  ( (\pi (Q_L))_t-\pi (Q_{*})_t)|^2  \,dx
\\
&+C \int_{ \Omega_{L,2\delta} \backslash  \Omega_{L,\delta} }\frac {|(\pi (Q_L))_t-Q_{*,t}|^2}{\delta^4}\( |\nabla ( Q_L-Q_*)|^2+ |\nabla ( \pi (Q_*)-\pi (Q_L))|^2\)\,dx.
\end{align*}
Note that
\begin{align*}
& \pi (Q_L)-\pi (Q_{*})=\na_{Q}\pi (Q_{\xi}) (Q_L-Q_{*}),\\
&(\pi (Q_L))_t-\pi (Q_{*})_t=\na_{Q}\pi(Q_{\xi})_t (Q_L-Q_{*}).
\end{align*}
When $Q_L$ approaches to $Q_*$, $\na_{Q}\pi (Q_{\xi})$ is close to the identity map $I$ and $\na_{Q}\pi(Q_{\xi})_t$ for small $t$.  Therefore
\begin{align*}
& |\na (\pi (Q_L) -\pi (Q_{*}))  |\leq C|\nabla (Q_L-Q_{*})| +C |\nabla Q_{\xi}| |Q_L-Q_{*}|.
\end{align*}
As $Q_L\to Q_*$, the term $\pi(Q_L)_t$ is close to $\pi(Q_*)_t$ and $\na_Q\pi(Q_\xi)_t$ is close to the identity map for small $t$. Note that $\na^2_{QQ}\pi(Q_\xi)_t$ is bounded. Then
\begin{align*}
|\na  ((\pi (Q_L))_t-\pi (Q_{*})_t)|&\leq |\na_Q\pi(Q_\xi)_t\na  (Q_L -Q_{*})|+|\na^2_{QQ}\pi(Q_\xi)_t||\na Q_\xi||Q_L -Q_{*}|
\\
&  \leq C|\nabla (Q_L-Q_{*})| + C|\nabla Q_{\xi}| |Q_L-Q_{*}|.
\end{align*}
Then the inequality \eqref{Q hat convergence} reads as
\begin{align*}
&\int_{\Omega} |\na  (\hat  Q_{L,t}-Q_{*,t})|^2 \,dx
\\
\leq& C\int_{ \Omega_{L,2\delta} }|\nabla (Q_L-Q_{*})|^2   +(|\nabla Q_L|^2 +|\nabla Q_*|^2)|Q_L-Q_{*}|^2\,dx
\\
\leq& C\int_{ \Omega }|\nabla (Q_L-Q_{*})|^2\d x  +C\(\int_{ \Omega\backslash \Sigma_\varepsilon}+\int_{\Sigma_\varepsilon }\)|\nabla Q_*|^2|Q_L-Q_{*}|^2\,dx.
\end{align*}
Here we employ the Egoroff theorem; i.e., for all $\varepsilon>0$, there exists a measurable subset $\Sigma_{\varepsilon} \subset \Omega$ such that
\begin{equation}\label{Egoroff}
| \Sigma_\varepsilon|\leq \varepsilon \mbox{ and }Q_L\to Q_*  \mbox{ uniformly on }\Omega\backslash\Sigma_{\varepsilon} \mbox{ as $L\to 0$}.
\end{equation}As $\varepsilon\to 0$ and $L\to 0$, we prove the claim that $\hat  Q_{L,t}\to Q_{*,t}$ strongly in $W^{1,2}_{Q_0}(\Omega; S_0)$.

 We observe that
\begin{align*}
& \int_{\Omega }|\tilde V_{p^k_{ij}}(Q_L,\na Q_L)\na_k(\hat Q_{L,t}-Q_L)_{ij}-\tilde V_{p^k_{ij}}(Q_*,\na Q_*)\na_k(Q_{*,t}-Q_*)_{ij}|\d x
\\
&\leq \int_{\Omega}| \tilde V_{p^k_{ij}}(Q_L,\na Q_L)||(\na_kQ_{L,t} -\na_kQ_{*,t})_{ij}+(\na_kQ_*-\na_kQ_L)_{ij}|\d x
\\
&+\(\int_{ \Omega\backslash \Sigma_\varepsilon}+\int_{\Sigma_\varepsilon }\) |\tilde V_{p^k_{ij}}(Q_L,\na Q_L)\na_k(Q_{*,t}-Q_*)_{ij}-\tilde V_{p^k_{ij}}(Q_*,\na Q_*)\na_k(Q_{*,t}-Q_*)_{ij}|\d x
\end{align*}
and
\begin{align*}
& \int_{\Omega }|\tilde V_{Q_{ij}}(Q_L,\na Q_L) (\hat Q_{L,t}-Q_L)_{ij} -\tilde V_{Q_{ij}}(Q_*,\na Q_*) (Q_{*,t}-Q_*)_{ij}|\d x
\\
\leq &\(\int_{ \Omega\backslash \Sigma_\varepsilon}+\int_{\Sigma_\varepsilon }\) |\tilde V_{Q_{ij}}(Q_L,\na Q_L) (\hat Q_{L,t}-Q_L)_{ij} -\tilde V_{Q_{ij}}(Q_*,\na Q_L) (\hat Q_{L,t}-Q_L)_{ij}|\d x
\\
&+\int_{\Omega } |\tilde V_{Q_{ij}}(Q_*,\na Q_L) (\hat Q_{L,t}-Q_L)_{ij}-\tilde V_{Q_{ij}}(Q_*,\na Q_*) (\hat Q_{*,t}-Q_*)_{ij}|\d x.
\end{align*}
Using the uniform convergence of $Q_L$ in $\Omega\backslash\Sigma_\varepsilon$ and strong convergence of $\hat Q_{L,t}, Q_L$ in $W^{1,2}_{Q_0}(\Omega,S_0)$,   we derive
\begin{align*}
&\lim_{L\to0}\int_{\Omega }|\tilde V_{Q_{ij}}(Q_L,\na Q_L) (\hat Q_{L,t}-Q_L)_{ij} -\tilde V_{Q_{ij}}(Q_*,\na Q_*) (Q_{*,t}-Q_*)_{ij}|\d x=0,
\\
 &\lim_{L\to0}\int_{\Omega }| \tilde V_{p^k_{ij}}(Q_L,\na Q_L)\na_k(\hat Q_{L,t}-Q_L)_{ij} -\tilde V_{p^k_{ij}}(Q_*,\na Q_*)\na_k(Q_{*,t}-Q_*)_{ij}|\d x=0.
 \end{align*}
As $L\to 0$, the estimate \eqref{eq sub of na f_B} yields
\begin{align*}\alabel{eq L varepsilon to 0}
&\int_{\Omega }\( \alpha \na_kQ_{*,ij} +\tilde V_{p^k_{ij}}(Q_*,\na Q_*)\)\na_k(Q_{*,t}-Q_*)_{ij}\d x
\\
&+\int_{\Omega }\tilde V_{ij}(Q_*,\na Q_*) (Q_{*,t}-Q_*)_{ij}\d x\geq 0.
\end{align*}
For each $\eta \in C_0^{\infty}(\Omega,S_0)$, we define
\begin{align*}
\varphi_{ij} (Q,\eta):=&
(s_+^{-1}Q_{jl}+\frac 13 \delta_{jl})\eta_{il}+(s_+^{-1}Q_{il}+\frac 13 \delta_{il})\eta_{jl}\alabel{test function 1}
\\
&-2(s_+^{-1}Q_{ij}+\frac 13 \delta_{ij})(s_+^{-1}Q_{lm}+\frac 13 \delta_{lm})\eta_{lm}.
\end{align*}
In view of \eqref{V1} and \eqref{V2},  we have
\begin{align*}
\lim_{t\to0}\frac{(Q_t-Q_*)}{t}=\varphi(Q_*,\eta),\quad \lim_{t\to0}\na \frac{(Q_t-Q_*)}{t}=\na\varphi(Q_*,\eta).
\end{align*}
For the estimate \eqref{eq L varepsilon to 0}, the limit in $t$ exists. Dividing \eqref{eq L varepsilon to 0} by $t$ then as $t\to 0^+$ and $t\to 0^-$, we have
\begin{align*}
	\int_{\Omega  } \(\alpha \na_kQ_{*,ij}+V_{p^k_{ij}}(Q_*,\na Q_*) \)\na_k  \varphi_{ij}(Q_*,\eta)+V_{Q_{ij}} (Q_*,\na Q_*)\varphi_{ij}(Q_*,\eta)\d x= 0.
\end{align*}
Repeating the same steps in \eqref{EL for one constant on S_*} and \eqref{EL for V on S_*}, we prove that $Q_*$ satisfies  \eqref{EL}.
\end{proof}

 \section{Smooth convergence of solutions}

In this section, we will prove Theorem 3. At first, we derive some key lemmas.

For any tensor $Q \in S_0$,  there exists a rotation $R(Q)\in SO(3)$     such that  $\tilde Q:=R^T(Q) Q R(Q)$ is  diagonal.
 Moreover, the space  $S^*$ has only three diagonal tensors so  for each $Q\in S_*$, we assume that
\begin{align*}\alabel{ROT Q}
 R^T(Q)QR(Q)= \begin{pmatrix}
		\frac{-s_+}{3}&0&0\\0&\frac{-s_+}{3}&0\\0&0&\frac{2s_+}{3}
	\end{pmatrix}:=Q^+.
\end{align*}

\begin{lemma}
	\label{lem f_B}
	For any $Q\in S_\delta$ and $\xi \in S_0$ with a sufficiently small $\delta>0$, the Hessian of the bulk density $f_B(Q)$ satisfies the following estimate
	\begin{align*}\alabel{eq Hessian}
		\sum_{i,j=1}^3\p_{\tilde Q_{ii}}\p_{\tilde Q_{jj}} f_B(\tilde Q)\xi_{ii}\xi_{jj}\geq \frac{\lambda}{2} ( \xi_{11}^2 +  \xi_{22}^2+ \xi_{33}^2),
	\end{align*}
	where $\lambda =\min\{3a,s_+ b \}>0$ and  $\tilde Q=R^T(Q) Q R(Q)$ is  diagonal.
\end{lemma}
\begin{proof} For a fixed $\pi(Q_0)\in S_*$, there exists a rotation $R(\pi(Q_0))\in   SO(3)$ in \eqref{ROT Q} such that $R^T(\pi(Q_0))\pi(Q_0) R(\pi(Q_0))=Q^+$. For $i=1, 2, 3$, we calculate  the first derivative of $f_B( \tilde  Q)$ by
	\begin{align*}
		\p_{ \tilde  Q_{ii}}f_B(\tilde  Q) =\(-a  \tilde  Q_{ii}-b \tilde  Q_{ik}\tilde  Q_{ki}+c \tilde  Q_{ii}| \tilde  Q|^2\).
	\end{align*}
	Then the second derivative of $f_B(\tilde  Q)$ with $i,j=1, 2, 3$ is
	\begin{align*}
		& \p_{\tilde  Q_{ii}}\p_{\tilde  Q_{jj}}f_B(\tilde  Q)= -a\delta_{ij}
		-2b\delta_{ij}\tilde  Q_{ii}+c(\delta_{ij}|\tilde  Q|^2+2\tilde  Q_{ii}\tilde  Q_{jj})
		.\alabel{2nd D}
	\end{align*}
	For the case of $i=j$ at $Q=Q_0$, $\tilde  Q=Q^+$, From the equality $\frac 23 cs^2_+=\frac 13bs_++a $ (c.f. \cite{MZ}), we find
	\begin{align*}
		\p_{\tilde  Q_{ii}}\p_{\tilde  Q_{ii}}f_B(\tilde  Q) =&-a -2\tilde  Q_{ii}b+( |\tilde  Q|^2+2\tilde  Q_{ii}^2)c=-(2\tilde  Q_{ii}-\frac{s_+}{3})b+2\tilde  Q_{ii}^2c.
	\end{align*}
	Then, at $\tilde  Q=Q^+$,  we have
	\begin{align}
		\p_{\tilde  Q_{11}}\p_{\tilde  Q_{11}}f_B(\tilde  Q)  =&\(s_+b+\frac{2s_+^2}{9}c\)=\frac 13a+\frac{10s_+}{9}b,
		\\
		\p_{\tilde  Q_{22}}\p_{\tilde  Q_{22}}f_B(\tilde  Q) =&\frac 13a+\frac{10s_+}{9}b,
		\\
		\p_{\tilde  Q_{33}}\p_{\tilde  Q_{33}}f_B(\tilde  Q)=&-s_+ b+\frac{8s_+}{9}c=\frac43a-\frac{5s_+}{9}b.
	\end{align}
	For the case of $i\neq j$, at $\tilde  Q=Q^+$, we observe that
	\begin{align}
		2\p_{\tilde  Q_{11}}\p_{\tilde  Q_{22}}f_B(\tilde  Q) =&4\tilde  Q_{11}\tilde  Q_{22}c=\frac{4s_+^2}{9}c=\frac23a+\frac{2s_+}{9}b,
		\\
		2\p_{\tilde  Q_{11}}\p_{\tilde  Q_{33}}f_B(\tilde  Q) =&4\tilde  Q_{11}\tilde  Q_{33}c=-\frac{8s_+^2}{9}c=-\(\frac43a+\frac{4s_+}{9}b\),
		\\
		2\p_{\tilde  Q_{22}}\p_{\tilde  Q_{33}}f_B(\tilde  Q) =&4\tilde  Q_{22}\tilde  Q_{33}c=-\frac{8s_+^2}{9}c=-\(\frac43a+\frac{4s_+}{9}b\).
	\end{align}
	In conclusion, using the fact that $\xi_{33}=-(\xi_{11}+\xi_{22})$, we have at $\tilde  Q=Q^+$
	\begin{align*}
		\p_{\tilde Q_{ii}}\p_{\tilde Q_{jj}} f_B(\tilde Q)\xi_{ii}\xi_{jj}
		=&\(\frac 13a+\frac{10s_+}{9}b\)(\xi_{11}^2+\xi_{22}^2)+\(\frac23a+\frac{2s_+}{9}b\)\xi_{11}\xi_{22}
		\\
		&+\(\frac43a-\frac{5s_+}{9}b \)\xi_{33}^2-\(\frac43a+\frac{4s_+}{9}b\)\xi_{33}(\xi_{11}+\xi_{22})
		\\
		=&  bs_+(\xi_{11}^2+\xi_{22}^2) + 3a\xi_{33}^2\geq \lambda(\xi_{11}^2 + \xi_{22}^2+\xi_{33}^2)
	\end{align*}
	with $\lambda =\min\{3a,s_+b\}>0$. Then
	\begin{align*}
		&\p_{\tilde Q_{ii}}\p_{\tilde Q_{jj}} f_B(\tilde Q)\xi_{ii}\xi_{jj}
		\geq\p_{\tilde Q_{ii}}\p_{\tilde Q_{jj}} f_B( Q^+)\xi_{ii}\xi_{jj}-C|\tilde Q-Q^+|\left|\sum_{i=1}^3\xi_{ii}\right|^2.
	\end{align*}
Due to the fact that $|\tilde Q-Q^+| = |Q-\pi(Q)|$, we prove \eqref{eq Hessian} for a sufficiently small $\delta>0$. 
\end{proof}

\begin{cor}For any $Q\in S_{\delta}$ with a sufficiently small $\delta>0$, there exists constants $C_1,C_2,C_3>0$ such that
\begin{align*}\alabel{f_B Distance}
C_1 \tilde f_B(Q)\leq|Q-&\pi(Q)|^2\leq C_2 \tilde f_B(Q),
\\
|Q-\pi(Q)|\leq&C_3 |g_B(Q)|\alabel{gb Qpi}.
\end{align*}
\end{cor}
\begin{proof}
It follows from the Taylor expansion of $\tilde f_B(Q)$ at $\pi(Q)$ that
\begin{align*}\alabel{f_B Taylor p2}
\tilde f_B(Q)=&\tilde f_B(\pi(Q))+\D_{Q_{ij}}f_B(\pi(Q))(Q-\pi(Q))_{ij}
\\
&+\p_{Q_{ij}}\p_{Q_{kl}} f_B(Q_\tau)(Q-\pi(Q))_{ij}(Q-\pi(Q))_{kl},
\end{align*}
where $Q_\tau$ is an intermediate point between $Q_L$ and $\pi(Q)$.

Since   $Q$  commutes with $\pi(Q)$ (c.f. \cite{NZ}), they can be simultaneously diagonalized.  Note that
$\tilde f_B(\pi(Q))=0$, $\D_{Q_{ij}}f_B(\pi(Q))=0$ and $Q_\tau$ is sufficiently close to $\pi(Q)$. Using Lemma \ref{lem f_B} with the fact that
 \begin{align*}
 	&\sum_{i,j=1}^3 \p_{\tilde  Q_{ii}}\p_{\tilde  Q_{jj}}f_B(\tilde Q)(\tilde Q-\tilde \pi(Q))_{ii}(\tilde Q-\tilde \pi(Q))_{jj}
 	\\
 	=&\p_{Q_{ij}}\p_{Q_{kl}} f_B(Q)(Q-\pi(Q))_{ij}(Q-\pi(Q))_{kl},
 \end{align*} we have
\begin{align}\label{eq f_B Q_tau}
\p_{Q_{ij}}\p_{Q_{kl}} f_B(Q_\tau)(Q-\pi(Q))_{ij}(Q-\pi(Q))_{kl}\geq \frac\lambda 2\sum_{i=1}^3(\tilde Q-Q^+)_{ii}^2.
\end{align}
Then we obtain
\[\tilde f_B(Q)\geq  \frac\lambda 2\sum_{i=1}^3(\tilde Q-Q^+)_{ii}^2= \frac\lambda 2 |Q-\pi(Q)|^2.\]
 The left-hand  side of \eqref{f_B Distance} is a direct consequence of \eqref{f_B Taylor p2} by using Young's inequality. Taking Taylor expansion of $ g_B(Q)$ at $\pi(Q)$ yields
 \begin{align*}
 	(g_B(Q))_{ij}=\p_{Q_{ij}}\p_{Q_{kl}} f_B(Q_{\tau_1})(Q-\pi(Q))_{kl}.
 \end{align*}
Multiplying both side by $(Q-\pi(Q))_{ij}$ and using \eqref{eq f_B Q_tau}, we obtain \eqref{gb Qpi}.
\end{proof}

From now on, for each $L>0$, let   $Q_L$ be a solution to the equation \eqref{MDEL}    and  assume that $Q_L$ is smooth and converges to $Q_*$ uniformly inside $\Omega\backslash \Sigma$, where $\Sigma$ is the singular set of $Q_*$.
For a sufficiently small $L$, $\dist(Q_L;S_*)\leq \delta$  inside $\Omega \backslash \Sigma$.

Set
\begin{align}\label{H}
	H(Q, \nabla Q):=&\alpha  \Delta Q +   \frac 12 {\na_k(V_{Q_{x_k}}+  V^T_{Q_{x_ k}})}-\frac 1 3  I \mbox{ tr} (\na_k  V_{Q_{x_k}})\\
	&-\frac 1 2 (  V_{Q}  +  V^T_{Q})+   \frac 1 3 I   \mbox{ tr} (V_{Q})
	\nonumber
\end{align}
and
\begin{align}\label{GB}
	g_B(Q)=  \(-aQ -b  \(Q Q -\frac 13I \tr(Q^2)\) -cQ \tr(Q^2)\).
\end{align}
Due to the fact that $\tilde Q:=R^T(Q) Q R(Q)$ is diagonal,
$g_B(\tilde Q)=R^T(Q)\, g_B(Q)R(Q)$
is also diagonal for a rotation $R(Q)\in SO(3)$.

Let $Q$ be differentiable in $\Omega$. Then
there exists a  set $\Sigma_Q$, which has measure zero, such that  $R(Q)$ is differentiable in $\Omega\backslash  \Sigma_Q$  (c.f. Corollary 2 \cite{MZ}, \cite{No}). Therefore, we have the following geometric identity of  rotations:

\begin{lemma}\label{Rota}   Assume that for any $x\in \Omega\backslash  \Sigma_Q$,     there exists a  differentiable  rotation $R(Q)$ such that both $R^T(Q)QR(Q)$ and $R^T(Q)h(Q)R(Q)$ are  diagonal.
	Then,   for each $i$, we have
	\begin{align*}\alabel{ROT}
		\nabla  \( R^T(Q)  h(Q) R (Q)\)_{ii} =\(R^T(Q)\nabla     h(Q) R (Q)\)_{ii}.
	\end{align*}
\end{lemma}
\begin{proof} Let $x_0$ be a   fixed point  in $\Omega\backslash \Sigma_Q$ and fix $i=1,2,3$.
	For $Q_0=Q(x_0)\in S_0$, there exists $R_0:=R(Q_0)\in SO(3)$ such that $R_0^TQ_0R_0$ is diagonal.  Denote $\tilde R(Q)=R_0^TR(Q)$ with
	$\tilde R(Q_0)=I$.
	Fix $Q_0\in S_0$, there is $R_0:=R(Q_0)\in SO(3)$ such that $R_0^TQ_0R_0$ and $R_0^Th(Q_0)R_0$ diagonal.  Denote $\tilde R(Q)=R_0^TR(Q)$, so
	$\tilde R(Q_0)=I$.
	Since $R(Q)\in SO(3)$, $R_{ki}(Q)R_{kj}(Q)=\delta_{ij}$.
	Then, for each $i$, we have at $x_0$
	\begin{align*}\alabel{ROT1}
		& \nabla   \( R^T(Q)  h(Q) R (Q)\)_{ii} =\nabla  \( R^T(Q) R_0R^T_0  h(Q)  R_0 R^T_0R (Q)\)_{ii}  \\
		= &  \sum_{k,l=1}^3 \nabla   \tilde R_{ki}(Q)(R^T_0  h(Q)  R_0)_{kl}\tilde R_{li}(Q) + \sum_{k,l=1}^3  \tilde R_{ki}(Q)(R^T_0 h(Q)  R_0)_{kl} \nabla_Q \tilde R_{li}(Q) \\
		&+   \sum_{k,l=1}^3 \tilde R_{ki}(Q) \nabla (R^T_0 h(Q)  R_0)_{kl}\tilde R_{li}(Q).
	\end{align*}

	Note that $R^T_0 h(Q_0) R_0$ is diagonal, $\tilde R_{ik}(Q_0)=\delta_{ik}$ and $\nabla \tilde R_{ii}(Q_0)=0$.
	It can be seen that the term $\nabla   \tilde R_{ki}(Q)(R^T_0Q_0 R_0)_{kl}\tilde R_{li}(Q)$ at $Q=Q_0$ is zero.
	Therefore \[\left .\nabla \( R^T(Q) h(Q)   R (Q)\)_{ii}\right |_{Q=Q_0}=\(R^T(Q_0)\nabla  h(Q)|_{Q=Q_0} R(Q_0)\)_{ii.}\]
	Since $x_0$ is any point,   we prove \eqref{ROT}.
\end{proof}

Denote the inner product by $\<A,B\>=A_{ij}B_{ij}$ for $A,B\in \mathbb M^{3\times 3}$. Using the above geometric identity, we have
\begin{lemma}\label{Var} Let $k$ and $l$ be two integers, for $x\in \Omega\backslash \Sigma_Q$, $Q= Q(x)$ and any smooth scalar function $\phi$, we have
	\begin{align*}\alabel{ROT gb}
		&\< \D^k (g_B(\tilde Q) \phi^2),R^T(Q)\D^{l+1} QR(Q)\>
		\\
		=&\D\< \D^k (g_B(\tilde Q) \phi^2),R^T\D^l QR\>-\< \D^{k+1} (g_B(\tilde Q) \phi^2) ,R^T(Q)\D^l QR(Q)\>.
	\end{align*}
\end{lemma}
\begin{proof} For $x_0\in \Omega\backslash \Sigma_Q$, $R(x)$ is differentiable in the neighborhood of $x_0$.
	Fixing $Q_0=Q(x_0)\in S_0$, there exists $R_0=R(Q(x_0))$ such that $R_0^TQ_0R_0$ is diagonal. Recall that $\tilde R(Q)=R_0^TR(Q)$, $\tilde R_{ik}(Q_0)=\delta_{ik}$ and $\nabla \tilde R_{ii}(Q_0)=0$.
	For any matrix $A$, let $A_D$ be the diagonal part of $A$ and $A_N$  the non-diagonal  part of $A$ such that
	$A=A_D+A_N$.
	Recall that $\D^k g_B(\tilde Q)$ and $\D^{k+1} g_B(\tilde Q)$ are diagonal. By employing an analogous argument in the proof of Lemma \ref{Rota}, we obtain
	\begin{align*}\alabel{Rot gb1}
		&\D\< \D^k (g_B(\tilde Q) \phi^2),\tilde R^T(Q)(R_0^T\D^{l} QR_0)_D\tilde R(Q)\>
		\\
		=&\< \D^{k+1} (g_B(\tilde Q) \phi^2) ,\tilde R^T(Q)(R_0^T\D^{l} QR_0)_D\tilde R(Q)\>
		\\
		&+\< \D^{k} (g_B(\tilde Q) \phi^2) ,\tilde R^T(Q)(R_0^T\D^{l+1} QR_0)_D\tilde R(Q)\>
		\\
		=&\< \D^{k+1} (g_B(\tilde Q) \phi^2),R^T(Q)\D^{l} QR(Q)\>
		+\< \D^{k} (g_B(\tilde Q) \phi^2) ,R^T(Q)\D^{l+1} QR(Q)\>
		\\
		=&\D\< \D^k (g_B(\tilde Q) \phi^2),R^T(Q)\D^l QR(Q)\>.
	\end{align*}Here we used that
\[\< \D^{k+1} (g_B(\tilde Q) \phi^2) ,\tilde R^T(Q)(R_0^T\D^{l} QR_0)_N\tilde R(Q)\>=0.\]
	Similarly, using \eqref{Rot gb1}, at $Q=Q_0$, we find
	\begin{align*}\alabel{ROT gb2}
		&\< \D^k (g_B(\tilde Q) \phi^2),R^T(Q)\D^{l+1} QR(Q)\>
		\\
		=&\left\< \D^k (g_B(\tilde Q) \phi^2) ,\(\tilde R^T(Q)\D(R_0^T\D^{l} QR_0)_D\tilde R(Q)\)_D\right\>
		\\
		=&\left\< \D^k (g_B(\tilde Q) \phi^2) ,\D\(\tilde R^T(Q)(R_0^T\D^{l} QR_0)_D\tilde R(Q)\)_D\right\>
		\\
		=&\D\< \D^k (g_B(\tilde Q) \phi^2) ,\tilde R^T(Q)(R_0^T\D^{l} QR_0)_D\tilde R(Q)\>
		\\
		&-\< \D^{k+1} (g_B(\tilde Q) \phi^2) ,\tilde R^T(Q)(R_0^T\D^{l} QR_0)_D\tilde R(Q)\>
		\\
		=&\D\< \D^k (g_B(\tilde Q) \phi^2) ,R^T(Q)\D^l QR(Q)\>-\< \D^{k+1} (g_B(\tilde Q) \phi^2) ,R^T(Q)\D^l QR(Q)\>.
	\end{align*}
	Since $x_0\in \Omega\backslash \Sigma_Q$ is arbitrary,  this completes the proof.
\end{proof}
Let $R_L=R(Q_L)$ be a rotation such that $R^T_LQ_LR$ is diagonal. Then  we have
\begin{lemma}\label{lem 2nd est} Let $x_0\in \Omega$ with some  $B_{r_0}(x_0)\subset \Omega\backslash \Sigma$ for a    sufficiently small $r_0$.
	Then,  for any $\phi  \in  C^\infty_0  ( B_{r_0}(x_0))$ and $Q_L\in S_\delta$ with sufficiently small $\delta$,  we have
	\begin{align}\label{eq 2nd}
		\int_{\Omega}\( |\D^2 Q_L|^2+ \frac{|(R^T_L\D Q_LR_L)_D|^2}{L} \)\phi^2\,dx\leq C\int_{\Omega}|\D Q_L|^2  |\D\phi|^2\,dx,
	\end{align}
	where $C$ is a constant independent of $L$. 
\end{lemma}

\begin{proof} Let $\varphi_\varepsilon$ be a cutoff function such that $\varphi_\varepsilon(x)=0$ for $\dist(x,\Sigma_{Q_L})\leq\varepsilon$ and $\varphi_\varepsilon(x)=1$ for $\dist(x,\Sigma_{Q_L})\geq 2\varepsilon$. Multiplying  \eqref{MDEL} by $\D (\phi^2 \varphi_{\varepsilon}^2 \D Q_L)$  yields
	\begin{align}\label{4.6}
		&\int_{\Omega} \<H(Q_L, \nabla Q_L) ,\,   \D (\phi^2\varphi_{\varepsilon}^2 \D Q_L)\> \,dx
		=\int_{\Omega} \<\frac 1L g_B(Q_L) ,\,   \D (\phi^2\varphi_{\varepsilon}^2 \D Q_L)\> \,dx.
	\end{align}
	Utilizing Lemma \ref{lem f_B} with 
	a sufficiently small  $\delta >0$, we derive
	\begin{align*}\alabel{K1}
		&\quad \<\D  g_B(\tilde Q_L), \nabla \tilde Q_L  \> = \D_{{\alpha}}[\p_{\tilde Q_{ii}}f_B(\tilde Q_L)]\nabla_{\alpha} (\tilde Q_L)_{ii} \\
		&=     \sum_{i,j}\p^2_{\tilde Q_{ii} \tilde Q_{jj}}f_B(\tilde Q_L) \nabla_{\alpha} (\tilde Q_L)_{ii} \nabla_{\alpha}( \tilde Q_L)_{jj}\geq \frac {\lambda} 2\sum_{i=1}^3|\nabla_{\alpha} (\tilde Q_L)_{ii}|^2=\frac {\lambda} 2|\nabla_{\alpha} \tilde Q_L|^2.
	\end{align*}
	Using  Lemma \ref{Rota} and \eqref{K1}, we have
	\begin{align*} \alabel{K}
		&\frac 1L\int_{\Omega} \< g_B(Q_L)\,,   \D (\phi^2\varphi_{\varepsilon}^2  \D Q_L)\>\,dx=\frac 1L\int_{\Omega} \< g_B(\tilde Q_L)\, , R^T_L\D (\phi^2\varphi_{\varepsilon}^2 \D Q_L) R_L\>\,dx
		\\
		=&  \frac 1L\int_{\Omega} \< g_B(\tilde Q_L)\, , \D (\phi^2\varphi_{\varepsilon}^2 R^T_L\D Q_L R_L) \>\,dx
		= -\frac 1L\int_{\Omega} \< \D g_B(\tilde Q_L)\, , \phi^2\varphi_{\varepsilon}^2 R^T_L\D Q_L R_L\>\,dx\\
		=&-\frac 1L\int_{\Omega} \< \D g_B(\tilde Q_L)\, ,  \phi^2 \varphi_{\varepsilon}^2 \D \tilde Q_L) \>\,dx
		\leq-\frac {\lambda} 2\int_{\Omega}\frac{|\nabla \tilde Q_L|^2}L\phi^2 \varphi_{\varepsilon}^2 \,dx
		\\
		\leq&-\frac {\lambda} 2\int_{\Omega}\frac{|(R^T_L\D Q_L R_L)_D|^2}L\phi^2 \varphi_{\varepsilon}^2 \,dx.
	\end{align*}
As $\varepsilon$ tends to zero, we observe that
	\begin{align*}
		&\lim_{\varepsilon\to 0}\int_{\Omega} \<H(Q_L, \nabla Q_L) ,\,   \D (\phi^2\varphi_{\varepsilon}^2 \D Q_L)\> \,dx
		\\
		=& -\lim_{\varepsilon\to 0}\int_{\Omega} \<\D ( H(Q_L, \nabla Q_L) ),\,   \phi^2\varphi_{\varepsilon}^2 \D Q_L \> \,dx=-\int_{\Omega} \<\D_iH(Q_L, \nabla Q_L) ,\,   \D_i Q_L\> \phi^2\,dx.
	\end{align*}
	It follows from using integrating by parts, \eqref{sec2 f_E} and Young's inequality that
	\begin{align*}\alabel{4.7}
		&-\int_{\Omega} \<\D_iH(Q_L, \nabla Q_L) ,\,   \D_i Q_L\> \phi^2\,dx
		\\
		=&-\int_{\Omega}\D_i\( \D_j\frac{\p f_{E,1}(Q_L,\D Q_L)}{\p {(Q_L)_{kl,j}}}-\frac{\p f_{E,1}(Q_L,\D Q_L)}{\p {(Q_L)_{kl}}}\)\D_i (Q_L)_{kl}\phi^2\,dx
		\\
		\geq& \int_{\Omega}\frac{\p^2f_{E,1}(Q_L,\D Q_L)}{\p( Q_L)_{kl,j}\p (Q_L)_{mn,r}}\D^2_{ir}(Q_L)_{mn}\D^2_{ij}(Q_L)_{kl}\phi^2\,dx
		\\
		&-C\int_{\Omega}(|\D^2 Q_L|+|\D Q_L|^2)|\D Q_L|^2\phi^2+(|\D^2 Q_L||\D Q_L|+|\D Q_L|^3)|\D\phi||\phi|\,dx
		\\
		\geq&\frac{\alpha} 4 \int_{\Omega}|\D^2 Q_L|^2\phi^2\,dx-C\int_{\Omega}|\D Q_L|^4 \phi^2+|\D Q_L|^2 |\D\phi|^2\,dx.
	\end{align*}
	Combining \eqref{K} with \eqref{4.7} yields
	\begin{align*}
		&\int_{\Omega}\(\frac{\alpha} 4|\D^2 Q_L|^2+\frac{\lambda}2 \frac{|(R^T_L\D Q_L R_L)_D|^2}L \)\phi^2\,dx\leq C \int_{\Omega}|\D Q_L|^4\phi^2  + |\D Q_L|^2|\D \phi|^2 \,dx.
	\end{align*}
	Integrating by parts and using Young's inequality, we deduce
	\begin{align*}
		&\int_{\Omega}|\D Q_L|^4 \phi^2\,dx= \int_{\Omega}\<\D Q_L, |\D Q_L|^2\D Q_L \> \phi^2\,dx  \\
		=&-\int_{\Omega}\< Q_L-Q_{L; x_0, r}, \D (|\D Q_L|^2\D Q_L) \> \phi^2\,dx
		\\
		&+2\int_{\Omega}\< Q_L-Q_{L; x_0, r}, |\D Q_L|^2\D Q_L\>  \phi \D \phi\,dx\\
		\leq& \frac 1 2 \int_{\Omega}|\D Q_L|^4 \phi^2\,dx + C\int_{\Omega}|Q_L-Q_{L; x_0, r}|^2 |\D^2 Q_L|^2\phi^2+  |\D Q_L|^2  |\D\phi|^2\,dx.
	\end{align*}
	Here $Q_{x_0, r}:=\fint_{B_{r}(x_0)}Q\,dx$. Note that
	\[|Q_L(x)-Q_{L; x_0, r}|\leq |Q_L(x)-Q_*(x)|+|Q_{L; x_0, r}- {Q_*}_{x_0, R}|+|Q_*(x)-{Q_*}_{x_0, R}| \]
	and for $x\in B_r(x_0)\subset \Omega\backslash \Sigma$, $Q_L(x)$ uniformly converges to $Q_*(x)$.
	For   a sufficiently small  $r_0$ and  $L$,  we see that
	\begin{align*}
		&\int_{\Omega}|\D Q_L|^4 \phi^2\,dx\leq
		C\int_{\Omega}|Q_L-Q_{L; x_0, r}|^2 |\D^2 Q_L|^2\phi^2\,dx +C\int_{\Omega} |\D Q_L|^2  |\D\phi|^2\,dx \\
		\leq& \frac\alpha4\int_{\Omega}|\D^2 Q_L|^2\phi^2\,dx +C\int_{\Omega} |\D Q_L|^2  |\D\phi|^2\,dx.
	\end{align*}
	Then we conclude that
	\begin{align*}
		&\int_{\Omega}  \( |\D^2 Q_L|^2+ \frac{|(R^T_L\D Q_LR_L)_D|^2}{L} \)\phi^2\, dx  \leq  C\int_{\Omega}|\D Q_L|^2  |\D\phi|^2\,dx.
	\end{align*}
\end{proof}

As an application of Lemma \ref{lem 2nd est}, we obtain a uniform Caccioppoli inequality for minimizer $Q_L$ as follows.

\begin{lemma}\label{lem 3nd est}
	 Let $x_0\in \Omega$ with  $B_{r_0}(x_0)\subset \Omega\backslash  \Sigma  $ for a sufficiently small $r_0>0$. Then for any $r\leq r_0$, we have
	\begin{align}
		\int_{B_{r/2}(x_0)} |\D Q_L|^2\,dx\leq\frac C{r^2}\int_{B_{r}(x_0)} | Q_L-  Q_{L; x_0, r}|^2\,dx,
	\end{align}
	where  $Q_{L; x_0, r}:=\fint_{B_{r}(x_0)}Q_L\,dx$ and $C$ is a constant independent of $L$.
\end{lemma}
\begin{proof}For two $s, t$ such that $\frac r 2\leq t< s\leq r$, choose a cutoff function $\phi\in C_0^{\infty} (B_s(x_0))$ such that $0\leq \phi \leq 1$, $\phi =1$ on
	$B_t$ and $|\D \phi |\leq C/(s-t)$.

Integrating by parts and 	using  Young's inequality, we have
	\begin{align*}
		\int_{\Omega} |\D Q_L(x)|^2 \phi^2\,dx=&- \int_{\Omega} \<\Delta Q_L, Q_L(x)-Q_{L; x_0, r}\>\phi^2\,dx\\
&- \int_{\Omega} \<\D Q, (Q(x)-Q_{x_0, r})\D \phi^2 \>\,dx\\
\leq
		&\frac 1 2\int_{\Omega} |\D Q_L(x)|^2 \phi^2 \,dx- \int_{\Omega} \<\Delta Q_L, Q_L(x)-Q_{L; x_0, r}\>\phi^2\,dx\\
		&+C  \int_{\Omega}|Q_L(x)-Q_{L; x_0, r}|^2  |\D\phi|^2\,dx.
	\end{align*}
	Then, by Young's inequality and Lemma \ref{lem 2nd est}, we obtain
	\begin{align*}
		\int_{B_t} |\D Q_L|^2 \,dx\leq
		&(s-t)^2\int_{\Omega} |\D^2 Q_L|^2 \,dx +\frac C {(s-t)^2} \int_{B_r(x_0)}|Q_L(x)-Q_{L; x_0, r}|^2 \,dx \\
		\leq& C_1\int_{B_s\backslash B_t}|\D Q_L|^2   \,dx+ \frac C {(s-t)^2} \int_{B_r(x_0)}|Q_L(x)-Q_{L; {x_0, r}}|^2 \,dx.
	\end{align*}
	Through the standard technique of 'filling hole', we have
	\begin{align*}
		\int_{B_t} |\D Q_L|^2 \phi^2\,dx\leq \theta \int_{B_s}|\D Q_L|^2   \,dx+ \frac C {(s-t)^2} \int_{B_r(x_0)}|Q_L(x)-Q_{L; x_0, r}|^2 \,dx
	\end{align*}
	for $\theta= \frac {C_1}{1+C_1} <1$ and two $s, t$ such that $r/2\leq t< s\leq r$. In view of Lemma 3.1 in Chapter V of \cite{Gi}, the relation \eqref{eq 2nd} follows.
\end{proof}

Using Lemma \ref{lem 3nd est}, we have   local uniform estimates on higher derivatives.
\begin{lemma}\label{k-ord}   Let $x_0\in \Omega$ with  $B_{r_0}(x_0)\subset \Omega\backslash \Sigma $.
Assume that there exists a constant $\varepsilon_0>0$  such that
	\begin{align}
		\int_{B_{r_0}(x_0)}|\D Q_L|^3\,dx\leq\varepsilon_0^3.
	\end{align}
	Then,  for any integer $k\geq 1$, there exist a constant $r_k\geq r_0/2$ and  a positive constant $C_k$  independent of $L$ such that
\begin{align*}
&\int_{B_{r_k}(x_0)} |\D^{k+1}Q_L|^2+\frac  1L |(R^T_L\D^k Q_LR_L)_D|^2 \,dx\leq C_k.
\alabel{k-ord estimates}
\end{align*}
\end{lemma}
\begin{proof} For simplicity of notations, we  denote $Q=Q_L$ and $R=R_L$.
The claim \eqref{k-ord estimates} is true for $k=1$.
At first, we show the case of $k=2$.

Assume that there exists a constant $C_1>0$ such that
\[\int_{B_{r_1}(x_0)}\(|\D Q|^4 +|\D^2 Q|^2+ \frac 1 L  |(R^T\D QR)_D|^2\)\,dx\leq C_1.\]
Let $\phi$ be a cutoff function in $C_0^{\infty}(B_{r_1}(x_0))$, where $r_2$ satisfies $\frac {r_0}2<r_2<r_1<r_0$ and $\phi=1$ in $B_{r_2}(x_0)$.  Let $\varphi_\varepsilon$ be another cutoff function such that $\varphi_\varepsilon(x)=0$ for $\dist(x,\Sigma_Q)\leq\varepsilon$ and $\varphi_\varepsilon(x)=1$ for $\dist(x,\Sigma_Q)\geq 2\varepsilon$.
We differentiate  \eqref{MDEL} twice and multiply by $\D^2  Q\phi^2\varphi_\varepsilon^2$ to get
\begin{align}\label{DH}
		& \int_{B_{r_0}(x_0)}\left \< \D^2_{\beta \gamma}H(Q, \D Q) , \D^2_{\beta \gamma}  Q\phi^2\varphi_\varepsilon^2 \right \> \,dx =\int_{B_{r_0}(x_0)}\left\<\frac 1 L  \D^2_{\beta \gamma} g_B( Q),\D^2_{\beta \gamma}  Q\phi^2\varphi_\varepsilon^2\right\>\,dx.
		\end{align}
	Applying Lemma \ref{Var} and Lemma \ref{lem f_B} to the right-hand side of \eqref{DH}, we find
\begin{align*}\alabel{G3}
	&  \frac 1L\int_{B_{r_0}(x_0)}\<  \D^2_{\beta \gamma} g_B(Q) ,  \D^2_{\beta \gamma} Q \phi^2\varphi_\varepsilon^2 \>\,dx
	=\frac 1L\int_{B_{r_0}(x_0)} \<  g_B(Q) ,  \D^2_{\beta \gamma} (\D^2_{\beta \gamma} Q \phi^2\varphi_\varepsilon^2) \>\,dx
	\\
	=&\frac 1L\int_{B_{r_0}(x_0)}\< g_B(\tilde Q) , R^T\D^2_{\beta \gamma} (\D^2_{\beta \gamma} Q \phi^2\varphi_\varepsilon^2) R\>\,dx
	\\
	=&\frac 1L\int_{B_{r_0}(x_0)}\< \D^2_{\beta \gamma}(g_B(\tilde Q) ), R^T\D^2_{\beta \gamma}  Q R \phi^2\varphi_\varepsilon^2 \>\,dx
	=\frac 1L\int_{B_{r_0}(x_0)}\< \D^2_{\beta \gamma}g_B(\tilde Q)  ,\D^2_{\beta \gamma}\tilde Q\>\phi^2\varphi_\varepsilon^2\,dx
	\\
	=& \frac 1 L  \int_{B_{r_0}(x_0)}\sum_{i,j}  \p^2_{\tilde Q_{ii}\tilde Q_{jj}}f_B(\tilde Q)  \D^2_{\gamma \beta} \tilde Q_{ii}\nabla^2_{\gamma \beta} \tilde Q_{jj} \phi^2\varphi_\varepsilon^2\,dx
	\\
	&+ \frac 1 L \int_{B_{r_0}(x_0)}\sum_{i,j,m}\p^3_{\tilde Q_{ii} \tilde Q_{jj}\tilde Q_{mm}}f_B(\tilde  Q)  \D_{\gamma} \tilde Q_{ii} \D_\beta \tilde Q _{jj} \nabla^2_{\gamma \beta}  \tilde Q_{mm} \phi^2\varphi_\varepsilon^2 \,dx
	\\
	\geq&\frac {\lambda} {4L }  \int_{B_{r_0}(x_0)}    |\D^2 \tilde Q|^2 \phi^2\varphi_\varepsilon^2\,dx-\frac{ C}{L}  \int_{B_{r_0}(x_0)}|\D\tilde  Q|^4 \phi^2 \varphi_\varepsilon^2 \,dx.
	\\
	\geq& \frac {\lambda} {4L }  \int_{B_{r_0}(x_0)}    |(R^T\D^2 QR)_D|^2 \phi^2\varphi_\varepsilon^2\,dx- C  \int_{B_{r_0}(x_0)}|\D Q|^2\frac{|(R^T\D  Q R)_D|^2}{L} \phi^2 \varphi_\varepsilon^2 \,dx,
\end{align*}
where we used that
\begin{align*}\alabel{D2 Q}
	&\int_{B_{r_0}(x_0)}    |\D^2 \tilde Q|^2 \phi^2\varphi_\varepsilon^2\,dx
	=-\int_{B_{r_0}(x_0)} \<\D (R^T QR)_D,\D(\D^2\tilde Q \phi^2\varphi_\varepsilon^2)\>\,dx
	\\
	=&\int_{B_{r_0}(x_0)} \<(R^T \D^2QR)_D \phi^2\varphi_\varepsilon^2,\D^2\tilde Q\>\,dx
	=\int_{B_{r_0}(x_0)}|(R^T\D^2 QR)_D|^2 \phi^2\varphi_\varepsilon^2\,dx.
\end{align*}
Observer that, for any fixed $\varepsilon>0$, $R(x)$ in \eqref{G3} is differentiable. Then it follows from \eqref{ROT Q} that $Q^+=R^T\pi(Q) R$ and
\begin{align*}\alabel{G3.1}
	&C  \int_{B_{r_0}(x_0)}|\D Q|^2\frac{|R^T\D  Q R|^2}{L} \phi^2 \varphi_\varepsilon^2 \,dx
	\\
	=&\frac{C}{L}\int_{B_{r_0}(x_0)}\<(R^T\D  Q R)_D,(R^T\D  Q R)_D\>|\D Q|^2\phi^2 \varphi_\varepsilon^2 \,dx
	\\
	=&\frac{C}{L}\int_{B_{r_0}(x_0)}\<\D( R^T Q R-Q^+)_D,\D (R^T Q R-Q^+)_D\>|\D Q|^2\phi^2 \varphi_\varepsilon^2 \,dx
	\\
	=&C  \int_{B_{r_0}(x_0)}|\D Q|^2\frac{|(R^T\D  (Q-\pi(Q)) R)_D|^2}{L} \phi^2 \varphi_\varepsilon^2 \,dx
	\\
	\leq&C  \int_{B_{r_0}(x_0)}|\D Q|^2\frac{|\D  (Q-\pi(Q))|^2}{L} \phi^2 \varphi_\varepsilon^2 \,dx.
\end{align*}
Letting $\varepsilon \to 0$ in \eqref{G3.1}, using  \eqref{MDEL} and the fact that $|Q-\pi(Q)|\leq C| g_B(Q)|$ in \eqref{gb Qpi}, we find
\begin{align*}\alabel{G3.2}
	&C  \int_{B_{r_0}(x_0)}|\D Q|^2\frac{|\D  (Q-\pi(Q))|^2}{L} \phi^2\,dx
	\\
	=&-\frac{C}{L} \int_{B_{r_0}(x_0)}\<(Q-\pi(Q),\D_\beta\big(\D_\beta  (Q-\pi(Q))|\D Q|^2 \phi^2\big)\>\,dx
	\\
	\leq&C\int_{B_{r_0}(x_0)}\frac{|Q-\pi(Q)|}{L}|\D^2 Q||\D Q|^2 \phi^2\,dx
	+C\int_{B_{r_0}(x_0)}\frac{|Q-\pi(Q)|}{L}|\D Q|^3 |\D\phi||\phi|\,dx
	\\
	\leq&C\int_{B_{r_0}(x_0)}|H(Q,\D Q)||\D^2 Q||\D Q|^2 \phi^2\,dx
	+C\int_{B_{r_0}(x_0)}|H(Q,\D Q)||\D Q|^3 |\D\phi||\phi|\,dx
	\\
	\leq&C\int_{B_{r_0}(x_0)}( |\D Q|^6 +|\D Q|^2|\D^2 Q|^2)\phi^2\,dx+C\int_{B_{r_0}(x_0)}(|\D Q|^4 +|\D^2 Q|^2)|\D\phi|^2\,dx.
\end{align*}
In view of \eqref{G3.1}-\eqref{G3.1}, we deduce \eqref{G3} to
\begin{align*}\alabel{G4}
	&\int_{B_{r_0}(x_0)}\left\<\frac 1 L  \D^2_{\beta \gamma} g_B( Q),\D^2  Q\phi^2\right\>\,dx
	\\
	\geq&  \frac {\lambda} {4L }  \int_{B_{r_0}(x_0)}    |(R^T\D^2 QR)_D|^2 \phi^2\,dx- C\int_{B_{r_0}(x_0)}( |\D Q|^6 +|\D Q|^2|\D^2 Q|^2)\phi^2\,dx
	\\
	&-C\int_{B_{r_0}(x_0)}(|\D Q|^4 +|\D^2 Q|^2)|\D\phi|^2\,dx.
\end{align*}
Applying \eqref{sec2 f_E} and Young's inequality to the left-hand side of \eqref{DH}, we obtain
\begin{align*}\alabel{2ord p4}
	&\int_{B_{r_0}(x_0)}\left \< \D^2_{\beta \gamma}H(Q, \D Q) , \D^2_{\beta \gamma}  Q\phi^2 \right \> \,dx
	\\
=&\int_{B_{r_0}(x_0)}    \D^2_{\beta \gamma}\( \D_k\frac{\p f_{E,1}(Q,\D Q)}{\p {Q_{ij,k}}}-\frac{\p f_{E,1}(Q,\D Q)}{\p {Q_{ij}}}\)  \D^2_{\beta \gamma} Q_{ij}     \phi^2\,dx
\\
\leq&  -\int_{B_{r_0}(x_0)} \frac{\p^2 f_E}{\p {Q_{ij,\gamma}}\p Q_{mn,l}} \D^3_{\alpha \beta l}Q_{mn} \D^3_{\alpha \beta \gamma} Q_{ij}\phi^2 \,dx
		\\
		&+C  \int_{B_{r_0}(x_0)} \left|\D\(\frac{\p^2 f_{E,1}(Q,\D Q)}{\p p\p Q} \D Q\)\right||\D^3 Q|\phi^2\,dx
		\\
		&+C  \int_{B_{r_0}(x_0)}\left|\D\frac{\p^2 f_{E,1}(Q,\D Q)}{\p p\p p} \right|  |\D^2 Q||\D^3Q|\phi^2\,dx
		\\
&+C  \int_{B_{r_0}(x_0)} \left|\D^2\frac{\p f_{E,1}(Q,\D Q)}{\p p}\right||\D^2 Q||\D\phi||\phi|+\left|\D^2\frac{\p f_{E,1}(Q,\D Q)}{\p Q}\right|  |\D^2 Q|\phi^2\,dx
		\\
		\leq & \int_{B_{r_0}(x_0)}\frac{\alpha}{4}|\D^3 Q|^2\phi^2\,dx
		+C\int_{B_{r_0}(x_0)}(|\D Q|^4 +|\D^2 Q|^2)|\D\phi|^2\,dx
		\\
& +C\int_{B_{r_0}(x_0)}( |\D Q|^6 +|\D Q|^2|\D^2 Q|^2)\phi^2\,dx.
\end{align*} 
Using H\"older's inequality, we have 
\begin{align*}\alabel{G5}
\int_{B_{r_0}(x_0)}(|\nabla^2Q|^2|\nabla Q|^2+|\nabla Q|^6)\phi^2\,dx\leq C\varepsilon_0 \int_{B_{r_0}(x_0)}|\nabla^3Q|^2 \phi^2+|\nabla^2 Q|^2|\nabla \phi |^2\,dx.
\end{align*}
Combining  \eqref{G4} with \eqref{G5} and
choosing $\varepsilon_0$ sufficiently small, we obtain
\begin{align*}
\int_{B_{r_2}(x_0)}|\nabla^3Q|^2 +  \frac 1 L |(R^T\D^2  QR)_D|^2\,dx\leq C_2.
\end{align*}
Set $r_k:=(1-\sum_{i=1}^k2^{-(i+1)})r_0>\frac {r_0}2$. For any $k\leq l$, we can assume that there is a constant $C_k$ such that
\begin{align*}\alabel{l-e}
\int_{B_{r_k}(x_0)} (|\D^{k+1}Q|^2+\frac  1L |(R^T\D^k  QR)_D|^2)\,dx\leq C_k.
\end{align*}
As a consequence of the Sobolev inequality, we have
\begin{align*}\alabel{l-s}
 \|\D^{k-1}Q\|_{L^{\infty}(B_{r_k}(x_0))} \leq C_k
\end{align*}
for any $k\leq l-1$.

Next, we prove it for $k=l+1$. Let $r_{l+1}$ be the constant satisfying
\[ \frac {r_0}2<r_{l+1}<r_{l}<\cdots <r_1 <r_0.\]  Let $\phi$ be a cutoff function in $C_0^{\infty}(B_{r_{l}}(x_0))$ with $\phi=1$ in $B_{r_{l+1}}(x_0)$.  We apply $\D^{l+1}$ to \eqref{MDEL} and multiply by $ \D^{l+1} Q\phi^2\varphi_\varepsilon^2 $ to have
\begin{align*}\alabel{Hk}
		& \int_{B_{r_0}(x_0)}\< \D^{l+1}H(Q, \nabla Q),\D^{l+1} Q\phi^2\varphi_\varepsilon^2\> \,dx
	=\int_{B_{r_0}(x_0)} \<\frac 1L \D^{l+1}g_B( Q),\D^{l+1} Q\phi^2\varphi_\varepsilon^2\>\,dx.
		\end{align*}
It follows from Lemma \ref{Var} that
	\begin{align*}\alabel{Gk3}
		&\frac 1L\int_{B_{r_0}(x_0)}\<\D^{l+1}g_B( Q),\D^{l+1} Q\phi^2\varphi_\varepsilon^2\>\,dx 
		\\
		=& (-1)^{l+1}\frac 1L\int_{B_{r_0}(x_0)}\<g_B( \tilde Q),R^T\D^{l+1}(\D^{l+1} Q\phi^2\varphi_\varepsilon^2)R\>\,dx
		\\
		=& \frac 1L\int_{B_{r_0}(x_0)}\<\D^{l+1}g_B( \tilde Q)\phi^2\varphi_\varepsilon^2,R^T\D^{l+1} QR\>\,dx
		\\
		=&(-1)^{l}\frac 1L\int_{B_{r_0}(x_0)}\<\D(\D^{l+1}g_B( \tilde Q)\phi^2\varphi_\varepsilon^2),\D\tilde  Q\>\,dx
		\\
		=&\frac 1L\int_{B_{r_0}(x_0)}\<\D^{l+1}g_B( \tilde Q),\D^{l+1}\tilde  Q\>\phi^2\varphi_\varepsilon^2\,dx
		\\
		=&\frac 1L\int_{B_{r_0}(x_0)}\sum_{i,j}\D^{l}\big(\p^2_{\tilde Q_{ii}\tilde Q_{jj}} f_B(\tilde Q)\D \tilde  Q_{jj}\big)\D^{l+1}\tilde  Q_{ii} \phi^2\varphi_\varepsilon^2\,dx.
	\end{align*} Using a similar argument in \eqref{D2 Q}, for $i\geq 3$, one can check that
\begin{align*}\alabel{Dk Q}
	&\int_{B_{r_0}(x_0)}  |\D^{i} \tilde Q|^2Z(x) \varphi_\varepsilon^2\,dx
	=(-1)^{i-1}\int_{B_{r_0}(x_0)}  \<\D\tilde Q,\D^{i-1}(\D^{i} \tilde QZ(x) \varphi_\varepsilon^2)\>\,dx
	\\
	=&\int_{B_{r_0}(x_0)}  \<\big(R^T\D^i QR)_DZ(x) \varphi_\varepsilon^2\big),\D^{i} \tilde Q\>\,dx
	=\int_{B_{r_0}(x_0)}  |(R^T\D^{i}  QR)_D|^2Z(x) \varphi_\varepsilon^2\,dx
\end{align*}for some scalar function $Z(x)$. 
Observe that $\p^j_{\tilde Q}f_B(\tilde Q)=0$ for $j\geq 5$. Applying Lemma \ref{lem f_B} with a sufficiently small $\delta$ and \eqref{Dk Q} to \eqref{Gk3}, we obtain
		\begin{align*}\alabel{Gk3.1}
		&\frac 1L\int_{B_{r_0}(x_0)}\<\D^{l+1}g_B( Q),\D^{l+1} Q\phi^2\varphi_\varepsilon^2\>\,dx 
		\\
		\geq&\frac 1{2 L}  \int_{B_{r_0}(x_0)}  |\D^{l+1} \tilde Q|^2\phi^2 \varphi_\varepsilon^2\,dx
		\\
		&-\frac C{ L}  \int_{B_{r_0}(x_0)}\sum_{\substack{\mu_1\leq \mu_2\leq l\\\mu_1+\mu_2=l+1}}|\p^3_{\tilde Q} f_B(\tilde Q)|^2|\D^{\mu_1} \tilde Q|^2|\D^{\mu_2} \tilde Q|^2\phi^2 \varphi_\varepsilon^2\,dx
		\\
		&-\frac C{ L}  \int_{B_{r_0}(x_0)}\sum_{\substack{\mu_1\leq \mu_2\leq\mu_3\leq l-1\\\mu_1+\mu_2+\mu_3=l+1}}|\p^4_{\tilde Q} f_B(\tilde Q)|^2|\D^{\mu_1} \tilde Q|^2|\D^{\mu_2} \tilde Q|^2|\D^{\mu_3} \tilde Q|^2\phi^2 \varphi_\varepsilon^2\,dx
		\\
		\geq&\frac 1{2 L}  \int_{B_{r_0}(x_0)}  |(R^T\D^{l+1}  QR)_D|^2\phi^2 \varphi_\varepsilon^2\,dx
		\\
		&-\frac C{ L}  \int_{B_{r_0}(x_0)}\sum_{\substack{\mu_1\leq \mu_2\leq l\\\mu_1+\mu_2=l+1}}|\p^3_{\tilde Q} f_B(\tilde Q)|^2|(R^T\D^{\mu_1}  QR)_D|^2|(R^T\D^{\mu_2}  QR)_D|^2\phi^2 \varphi_\varepsilon^2\,dx
		\\
		&-\frac C{ L}  \int_{B_{r_0}(x_0)}\sum_{\substack{\mu_1\leq \mu_2\leq\mu_3\leq l-1\\\mu_1+\mu_2+\mu_3=l+1}}|\p^4_{\tilde Q} f_B(\tilde Q)|^2|(R^T(\D^{\mu_1}  Q)R)_D|^2
		\\
		&\times|(R^T\D^{\mu_2}  QR)_D|^2|(R^T\D^{\mu_3}  QR)_D|^2\phi^2 \varphi_\varepsilon^2\,dx.
	\end{align*}
		As $\varepsilon$ tends to zero, we obtain from \eqref{l-s} and \eqref{Gk3.1} that
		\begin{align*}\alabel{Gk}
		&\lim_{\varepsilon\to 0}\frac 1L\int_{B_{r_0}(x_0)}\<\D^{l+1}g_B( Q),\D^{l+1} Q\phi^2\varphi_\varepsilon^2\>\,dx 
		\\
		\geq& \frac 1{4 L}  \int_{B_{r_0}(x_0)}  |(R^T\D^{l+1} Q  R)_D|^2\phi^2 \,dx
		-\frac C{L}  \int_{B_{r_0}(x_0)}  |(R^T\D Q  R)_D|^2|(R^T\D^{l} Q  R)_D|^2\phi^2 \,dx
		\\
		&-\frac C{L}  \int_{B_{r_0}(x_0)}  |(R^T\D^2 Q  R)_D|^2|(R^T\D^{l-1} Q  R)_D|^2\phi^2 \,dx-C_k
		\\
		\geq& \frac 1{4 L}  \int_{B_{r_0}(x_0)}  |(R^T\D^{l+1} Q  R)_D|^2\phi^2 \,dx-C_k.
	\end{align*}
Using Young's inequality and integration by parts,  we have
\begin{align*}\alabel{Ek}
	&\int_{B_{r_0}(x_0)}\< \D^{l+1}H(Q, \nabla Q),\D^{l+1} Q\phi^2\> \,dx
	\\
	=&-\int_{\R^3} \frac{\p^2 f_E}{\p {Q_{ij,\gamma}}\p Q_{mn,\beta}} \D^{l+1}\D_{ \beta }Q_{mn} \D^{l+1}\D_{ \gamma} Q_{ij}\phi^2 \,dx
	\\
	&+C  \int_{B_{r_0}(x_0)} \left|\D^{l+1}\frac{\p f_{E,1}(Q,\D Q)}{\p p}\right||\D^{l+1} Q||\D\phi||\phi|\,dx
	\\
	&+C  \int_{B_{r_0}(x_0)} \left|\D^{l+1}\frac{\p f_{E,1}(Q,\D Q)}{\p Q}\right|  |\D^{l+1} Q|\phi^2\,dx\\
	&+C  \int_{B_{r_0}(x_0)} \left|\D^l\(\frac{\p^2 f_{E,1}(Q,\D Q)}{\p p\p Q} \D Q\)\right||\D^{l+2} Q|\phi^2\,dx
	\\
	&+C  \int_{B_{r_0}(x_0)}\left|\D^{l-1}\(\D\frac{\p^2 f_{E,1}(Q,\D Q)}{\p p\p p} \D^2Q\)\right|  |\D^{l+2} Q|\phi^2\,dx	
	\\
 \leq&-\int_{B_{r_0}(x_0)}\frac{\alpha}{4}|\D^{l+2} Q|^2\phi^2\,dx
 \\
 &+ C\int_{B_{r_0}(x_0)} \sum_{\substack{\mu_1+\cdots + \mu_i=l+2\\  \mu_1\leq \cdots \leq \mu_{i-1} <l+2} }|\D^{\mu_1}Q|^2|\D^{\mu_2} Q|^2\cdots |\D^{\mu_{i-1}}Q|^2 |\D^{\mu_i}\phi|^2 \,dx.
\end{align*}
In view of \eqref{l-e}-\eqref{l-s}, we have
\begin{align*}\alabel{Ek1}
	&  C\int_{B_{r_0}(x_0)} \sum_{\substack{\mu_1+\cdots + \mu_i=l+2\\  \mu_1\leq \cdots \leq \mu_{i-1} <l+2} }|\D^{\mu_1}Q|^2|\D^{\mu_2} Q|^2\cdots |\D^{\mu_{i-1}}Q|^2 |\D^{\mu_i}\phi|^2 \leq C_l.
\end{align*}
Combining \eqref{Ek} with \eqref{Ek1}, our claim \eqref{k-ord estimates} follows for $k=l+1$.
\end{proof}
Now we give a proof of Theorem  \ref{Theorem 3}.

\begin{proof}
For any $x_0\in  \Omega\backslash \Sigma$, 	let $B_{2R_0}(x_0)$ be a ball  such that   $B_{2r_0}(x_0)\subset \Omega\backslash \Sigma$. From Lemma \ref{lem 3nd est},  we deduce the following estimate
	\begin{align}
		&\frac 1{ r_0}\int_{B_{\frac 43 r_0}(x_0)} |\D Q_L|^2\,dx\leq C\fint_{B_{2r_0}(x_0)} | Q_L-  {Q_L}_{x_0, 2r_0}|^2\,dx\\
&\leq C\sup_{x\in B_{2r_0}(x_0)}|Q_L(x)-Q_{L; x_0,\rho}|^2.\nonumber
	\end{align}
	It follows from \eqref{eq 2nd} that
\[r_0\int_{B_{r_0}(x_0)}  |\D^2 Q_L|^2\,dx\leq \frac C {r_0}\int_{B_{\frac 43 r_0}(x_0)} |\D Q_L|^2 \,dx. \]
As a consequence of the Gagliardo-Nirenberg interpolation (c.f. \cite{FHM}), we have
	\begin{align*}
		\int_{B_{r_0}(x_0)} |\D Q_L|^3\,dx\leq & C\left( \int_{B_{r_0}(x_0)}r_0|\D^2 Q_L|^2+\frac 1 {r_0}|\D Q_L|^2 \,dx\right)^{3/2}\\
&+C\left(r_0^{-1}\int_{B_{r_0}(x_0)}|\D Q_L|^2\,dx\right)^{3/2}\leq   \varepsilon_0^3.
	\end{align*}
	Using Lemma 4.5 with any $k\geq 1$, we obtain
	\begin{equation}\label{con1.0}
		\|Q_L\|_{W^{k,2}_\loc(\Omega\backslash\Sigma)}<C_k.
	\end{equation}
	Then,   $Q_L$ converges smoothly to  $Q_*$ in $\Omega\backslash\Sigma$.
\end{proof}

\section{The  Landau-de Gennes  density through the Oseen-Frank  density}

In this section,  we will obtain a new form of the Landau-de Gennes  energy density through the Oseen-Frank  density. Under the condition \eqref{L}, it was shown in \cite  {MGKB} that for each $Q =s_+ (u\otimes u-\frac 13 I)\in S_*$, one has
\[W(u,\na u) =f_E(Q,\nabla Q). \]
Assuming   the strong Ericksen condition
\begin{equation}\label{Er}	k_2>|k_4|,\quad k_3>0,\quad 2k_1>k_2+k_4,
\end{equation}
it was pointed out in \cite{HM} (see also \cites{Er2,Ba,FHM}) that there are positive constants $\lambda$ and $C$ such that the density $W(u,\na u)$ is equivalent to a new form that  $\widetilde W(u,p)$  satisfies
\begin{equation*}
	\lambda |p|^2\leq \widetilde W(u,p)\leq C(1+|u|^2) |p|^2
\end{equation*}
for any $u\in\R^3$ and any $p\in\mathbb M^{3\times 3}$. However, there seems   no  reference for an explicit form of $\widetilde W(u,\na u)$, so we give an explicit form $\widetilde W(u,\na u)$ here. For $u\in \R^3$, it can be checked that
	\begin{align*}\alabel{V last}
		&2|\D  u|^2-|\curl u|^2-(\div  u)^2
		\\
		=&(\D_1 u_1+\D_2 u_2)^2+(\D_2 u_2+\D_3 u_3)^2+(\D_1 u_1+\D_3 u_3)^2
		\\
		&+(\D_2u_3+\D_3u_2)^2+(\D_1u_3+\D_3u_1)^2+(\D_1u_2+\D_2u_1)^2.
\end{align*}

\begin{lemma} \label{Lemma 3.1}
	Assume the Frank constants $k_1, \cdots ,k_4$ satisfy
	\begin{align*}\alabel{new k}
	 k_2>|k_4|,\quad \min \{k_1, k_3\}>\frac 12 (k_2+k_4).
	\end{align*}
Then
	 the density  $W(u, \nabla u)$ of  the form (\ref{OF density}) for each $u\in S^2$   is equivalent to the new form
	\begin{align*}\alabel{W}
		\widetilde W(u,\nabla u)=&\frac{ \tilde\alpha}{2} |\nabla u|^2+\frac{2k_1- k_2-k_4- \alpha}{4} (\div u)^2
		\\
		&+\frac{k_2-k_4- \alpha}{4} (u\cdot\curl u)^2
		+\frac{2k_3-k_2-k_4- \alpha}{4}|u\times \curl u|^2
		\\
		& +\frac{k_2+k_4- \alpha}{4}  \sum_{i\neq j}\((\D_i u_i+\D_j u_j)^2+(\D_iu_j+\D_ju_i)^2\)
	\end{align*}
where $ \alpha =\min\{k_2-|k_4|,2k_1- k_2-k_4,2k_3- k_2-k_4\}>0$.
\end{lemma}
\begin{proof}
	Note that $W(u, \nabla u)$ is rotational invariant (c.f. \cite{Ho1}); i.e., for each  $R\in SO(3)$, $\tilde x=R(x-x_0)$ and $\tilde u = R  u(x)=R u$. Then we have
	\[
	W(\tilde u, \tilde \na \tilde  u)= W(Ru, R\na  uR^T)=W(u,\na u) .
	\]
	Then for any $u\in S^2$, we can find some $R=R(u(x_0))\in SO(3)$  at each point $x_0\in \Omega$ such that
	\[\tilde u(0):=R u(x_0) = (0,0,1)^T.\]
	Using the relation
	\[\frac{\p \tilde u_3}{\p \tilde x_i} = -(\tilde u_1 \frac{\p \tilde u_1}{\p \tilde x_i}+  \tilde u_2 \frac{\p \tilde u_2}{\p \tilde x_i})=0 \] for all $i=1,2,3$, we evaluate four terms of the Oseen-Frank energy density at $\tilde x_0$
	\begin{align*}\alabel{relation D u}
		(\tilde \na \cdot \tilde u)^2=&(\tilde \na_1 \tilde u_1+\tilde \na_2 \tilde u_2)^2,&(\tilde u\cdot \curl \tilde u)^2=& (\tilde \na_1\tilde u_2-\tilde \na_2\tilde u_1)^2 ,
		\\
		|\tilde u\times  \curl \tilde u|^2
				=&|\tilde \na_3\tilde u_1|^2+|\tilde \na_3\tilde u_2|^2,
		&(\tr(\tilde \na \tilde u)^2-(\tilde \na \cdot \tilde u)^2)=&2\tilde \na_1\tilde u_2\tilde \na_2\tilde u_1-2\tilde \na_1 \tilde u_1\tilde \na_2 \tilde u_2.
	\end{align*}
	Substituting above identities into the density, we have
	\begin{align*}\alabel{eq W}
		W(\tilde u,\tilde \na \tilde u)
		=&\frac {k_1}2(\tilde \na_1 \tilde u_1+\tilde \na_2 \tilde u_2)^2 +\frac{k_2}2 (|\tilde \na_1\tilde u_2|^2+|\tilde \na_2\tilde u_1|^2)
		\\
		& +\frac {k_3}2(|\tilde \na_3\tilde u_1|^2+|\tilde \na_3\tilde u_2|^2)+  {k_4}  \tilde \na_1\tilde u_2\tilde \na_2\tilde u_1-(k_2+k_4)(\tilde \na_1 \tilde u_1\tilde \na_2 \tilde u_2)
		\\
		=&\frac { \tilde\alpha}2 |\tilde \na \tilde u|^2+\frac{2k_1-k_2-k_4- \tilde\alpha}{4}(\tilde \na_1 \tilde u_1+\tilde \na_2 \tilde u_2)^2
		\\
		&+\frac{k_2+k_4- \tilde\alpha}{4}(\tilde \na_1 \tilde u_1-\tilde \na_2 \tilde u_2)^2+\frac {k_2-k_4- \tilde\alpha}2(|\tilde \na_1\tilde u_2|^2+|\tilde \na_2\tilde u_1|^2)
		\\
		&+\frac {(k_3- \tilde\alpha)}2(|\tilde \na_3\tilde u_1|^2+|\tilde \na_3\tilde u_2|^2)+\frac {k_4}2(\tilde \na_1\tilde u_2+\tilde \na_2\tilde u_1)^2
		\\
		=&\frac { \tilde\alpha}2| \na  u|^2+\frac{2k_1- k_2-k_4- \tilde\alpha}{4} (\div u)^2+\frac{k_2-k_4- \tilde\alpha}{4} (u\cdot\curl u)^2
		\\
		& +\frac {k_3- \tilde\alpha}2|u\times \curl u|^2+\frac{k_2+k_4- \tilde\alpha}{4}\(|\tilde \na_1 \tilde u_1-\tilde \na_2 \tilde u_2|^2+|\tilde \na_1\tilde u_2+\tilde \na_2\tilde u_1|^2\),
	\end{align*}
	where $ \tilde\alpha$ is a positive constant due to the strong Ericksen condition (\ref {Er}). Using \eqref{relation D u}, we find
	\begin{align*}
		&|\tilde \na_1 \tilde u_1-\tilde \na_2 \tilde u_2|^2+|\tilde \na_1\tilde u_2+\tilde \na_2\tilde u_1|^2+|\tilde \na_3\tilde u_1|^2+|\tilde \na_3\tilde u_2|^2
		\\
		=&|\tilde\D \tilde u|^2+\tr(\tilde\D \tilde u)^2-(\tilde\D\cdot \tilde u)^2=2|\D  u|^2-|\curl u|^2-(\div  u)^2.
	\end{align*}
	If we further assume that $2k_3>k_2+k_4$, then we can rewrite \eqref{eq W} into
	\begin{align*}\alabel{eq W1}
	\widetilde W( u,  \na   u)=&\frac{\tilde\alpha} 2| \na  u|^2+\frac{2k_1- k_2-k_4- \alpha}{4} (\div u)^2
	\\
	&
	+\frac{k_2-k_4- \alpha}{4} (u\cdot\curl u)^2+\frac{2k_3-k_2-k_4- \alpha}{4}|u\times \curl u|^2
	\\
	& +\frac{k_2+k_4- \alpha}{4}\(2|\D  u|^2-|\curl u|^2-(\div  u)^2\).
	\end{align*}
	From \eqref{V last}, we prove \eqref{W}.
\end{proof}
It is clear that  the new form  $\widetilde W(u,p)$ in \eqref{W} with $p=\nabla u$ satisfies
\begin{equation*}
	 \frac{\tilde\alpha} 2  |p|^2\leq \widetilde W(u,p)\leq C(1+|u|^2)|p|^2
\end{equation*}
for all $u\in\R^3$ and  $p\in\mathbb M^{3\times 3}$.

Through the relation \eqref{W},  we can have  the  new Landau-de Gennes  energy density satisfying the coercivity in the following:

\begin{prop}\label{Theorem 4}  Assume that $\hat L_1$, $\hat L_2$, $\hat L_3$ and $\hat L_4$ satisfy the condition
	\begin{align*}\alabel{Er1}
	\hat L_1:=&s_+^{-2}\frac{2k_1- k_2-k_4}{2},
	&\hat L_2:=&s_+^{-2}\frac{2k_3-k_2-k_4}{2}
	\\
	\hat L_3:=&s_+^{-2}\frac{k_2-k_4}{2},
	&\hat L_4:=&s_+^{-2}\frac{k_2+k_4}{2}.
	\end{align*}
	Then for each $Q\in S_*$, we obtain
	\begin{align*}
	f_{E,2}(Q,\nabla Q):=&\frac{ \alpha}2|\nabla Q|^2 + \frac{\hat L_1-\tilde \alpha}{2}
	\sum_{i=1}^3 \Big((s_+^{-1}Q +\frac13I)_{ij}(\na\cdot Q_j)\Big)^2
	\\
	& +\frac{\hat L_2-\tilde \alpha}{2}
	\left| (s_+^{-1}Q+\frac13I)_i\times\curl Q_{i}\right|^2
	\\
	&+ \frac{\hat L_3-\tilde \alpha}{2}	\sum_{i=1}^3\( (s_+^{-1}Q +\frac13I)_{ij}(\curl Q_j)_i\)^2
	\\
	&+\frac{\hat L_4-\tilde \alpha}{2}\sum_{i\neq j}\sum_{k=1}^3 \Big( (\nabla_iQ_{ik} +\nabla_jQ_{jk})^2+(\nabla_iQ_{jk}+\nabla_jQ_{ik})^2 \Big),
	\end{align*}
	where $Q_i$ is the $i$-th column of the $Q$ matrix and $\tilde \alpha$ is given by
		\begin{align*}\alabel{eq alpha}
		& \alpha=\min\{\hat L_1,\hat L_2,\hat L_3,\hat L_4 \} >0
		.\end{align*}
\end{prop}

\begin{proof}
	Due to the fact that $|u|^2=1$, a direct calculation yields
	\[\nabla_ku_i =u_j\nabla_k (u_i u_j)=s_+^{-1}(u_1\nabla_kQ_{i1}+u_2\nabla_kQ_{i2}+u_3\nabla_kQ_{i3})=s_+^{-1}u_j\nabla_kQ_{ij}.
	\]
	One can verify  that
	\begin{align*}
		u_l\nabla_ku_i =s_+^{-1}(s_+^{-1}Q +\frac13I)_{lj}\nabla_kQ_{ij}.\alabel{u and Q}
	\end{align*}
	Here we used the fact that $|Q|=\sqrt{\frac 2 3}s_+$.
	Then we can derive $f_E(Q,\D Q)$ from \eqref{W} that
	\begin{align}\label{u into Q I2}
		&s_+^2(\div u)^2=s_+^2\sum_i (u_i\div u)^2=\sum_i \Big((s_+^{-1}Q +\frac13I)_{ij}(\na\cdot Q_j)\Big)^2,
		\\
		&s_+^2|u\times \curl u|^2=s_+^2|u\times (s_+^{-1}u_j(\curl Q_{j}))|^2=\left| (s_+^{-1}Q+\frac13I)_i\times\curl Q_{i}\right|^2,
		\\
		&s_+^2(u\cdot \curl u)^2=s_+^2(s_+^{-1}u_iu_j(\curl Q_j)_i)^2=\(  (s_+^{-1}Q +\frac13I)_{ij}(\curl Q_j)_i\)^2.\end{align}
	It then follows from \eqref{V last} and \eqref{u and Q} that
	\begin{align*}\alabel{I4}
		&s_+^2\(2|\D  u|^2-|\curl u|^2-(\div  u)^2\)
		\\
		=&\frac12 \sum_{l,k=1}^3\sum_{i\neq j}((s_+^{-1}Q +\frac13I)_{lk})^2\((\nabla_iQ_{ik} +\nabla_jQ_{jk})^2+(\nabla_iQ_{jk}+\nabla_jQ_{ik})^2 \).
	\end{align*}
	Substituting the identities  \eqref{u into Q I2}-\eqref{I4} into the equation \eqref{W}, we complete a proof.
\end{proof}

\medskip
\noindent{\bf Acknowledgements:}
We would like to thank   Professors John Ball  and Arghir Zarnescu for  their useful comments on our early version.

\medskip

\noindent{\bf Data availability.} No datasets were generated or analyzed during the current study.

\end{document}